\documentclass[12pt]{amsart}
\usepackage{amsmath,amsthm,amscd,amsfonts,amssymb,graphicx,color,mathabx,enumitem,mathtools,comment,amsrefs}

    \makeatletter
    \def\l@subsection{\@tocline{2}{0pt}{2.9pc}{5pc}{}}
    \def\l@subsubsection{\@tocline{2}{0pt}{5pc}{7.5pc}{}}
    \makeatother
\usepackage[margin=2.8cm]{geometry}

\usepackage[bookmarksnumbered, colorlinks, plainpages]{hyperref}
    \hypersetup{colorlinks=true, linkcolor=blue, citecolor=magenta, filecolor=magenta, urlcolor=cyan}
\allowdisplaybreaks
    \mathtoolsset{showonlyrefs,showmanualtags} 
\numberwithin{equation}{section}
\newtheorem{theorem}[equation]{Theorem}
\newtheorem{lemma}[equation]{Lemma}
\newtheorem{proposition}[equation]{Proposition}

\theoremstyle{definition}
\newtheorem{definition}[equation]{Definition}

\newtheorem{remark}[equation]{Remark}
\newtheorem{question}[equation]{Question}
\usepackage{scalerel} 
\DeclareMathOperator*{\Expectation}{\scalerel*{\mathbb{E}}{\textstyle\sum}}

\begin{document}
\title[Integer Cantor Sets]{Integer Cantor Sets:\\ Arithmetic Combinatorial Properties}

\author[Burgin]{Alex Burgin} 
    \address{School of Mathematics, Georgia Institute of Technology, Atlanta GA 30332, USA}
    \email{alexander.burgin@gatech.edu}
    \thanks{Research of AB supported in part by the Department of Education Graduate Assistance in Areas of National Need program at the Georgia Institute of Technology (Award \# P200A240169)}

\author[Fragkos]{Anastasios Fragkos}
\address{ School of Mathematics, Georgia Institute of Technology, Atlanta GA 30332, USA}
    \email {anastasiosfragkos@gatech.edu}

\author[Lacey]{Michael T. Lacey} 
    \address{ School of Mathematics, Georgia Institute of Technology, Atlanta GA 30332, USA}
    \email {lacey@math.gatech.edu}
    \thanks{Research of MTL and AB supported in part by a grant from the US National Science Foundation, DMS-2247254.}

 \author[Mena]{Dario Mena}
    \address{Centro de Investigaci\'{o}n en Matem\'{a}tica Pura y Aplicada \\ Escuela de Matem\'{a}tica, Universidad de Costa Rica}
    \email{dario.menaarias@ucr.ac.cr}
    \thanks{Research of DM partially supported by grant C1012 from Universidad de Costa Rica}

\author[Reguera]{Maria Carmen Reguera} 
 \address{Universidad de M\'alaga}
    \email{m.reguera@uma.es}
    \thanks{Research of MCR supported by the Spanish Ministry of Science and Innovation through the projects RYC2020-030121-IAEI/10.13039/501100011033/ and PID2022-136619NB-I00 funded by MCIN/AEI/10.13039/501100011033/FEDER, UE}

\begin{abstract}
    Cantor sets of integers, namely the integers whose base $b$ digits all lie in a fixed set $D$, have a rich set of arithmetic combinatorial properties. 
    Assuming that $0$ is an allowed digit, we show the following. 
    (a) These sets are van der Corput, and hence intersective, and polynomials evaluated along them are intersective, with a positive density of return times. 
    Allowing signed digits, we give examples with power savings in the van der Corput property, of arbitrarily small dimension. 
    (b) A real polynomial evaluated at the pair consisting of an element of the Cantor set and its sum of digits is uniformly distributed modulo $1$, uniformly over intervals of indices, exactly when the polynomial has an irrational non-constant coefficient. 
    For a polynomial map into $\mathbb{R}^2$, the correct condition is instead that every non-trivial integer combination of the two coordinates be irrational. 
    We also give sufficient conditions, and separately some necessary conditions, for these pairs to be uniformly distributed in residue classes. 
    (c) Polynomial mean ergodic theorems hold along Cantor sets. 
    (d) If the digits are at most $(b-1)/2$, the dilates of the Cantor set have metric Poissonian pair correlation, with a bound on the Hausdorff dimension of the exceptional set in terms of the additive energy of $D$. 
\end{abstract}

\maketitle

\tableofcontents

\section{Introduction} % (fold)
\label{sec:introduction}

We establish arithmetic combinatorial properties of Cantor sets of integers.  
The most accessible example of such a set of integers would be the integer analog of the classical middle third Cantor set: all non-negative integers $n$ that, when written in base $3$, do not have $1$ as a digit.   
The properties we address in this paper are  (a) intersective and van der Corput properties, 
(b) uniform distribution for polynomials evaluated along Cantor sets, 
(c) mean polynomial ergodic theorems,  and 
(d) metric pair correlation.  

Let $\mathcal{C}_{b,D}$ denote a \emph{classical integer Cantor set} defined by a base $b\geq 3$ 
and a collection of digits 
$ D \subset \{0,1,\ldots,b-1\}$. 
That is, 
\begin{equation}
\mathcal{C}_{b, D}  \coloneqq 
\Bigl\{ 
 \sum_{j=0}^k d_j b ^{j} 
 \colon  k\in \mathbb{N}_0 ,\ d_j \in  D
\Bigr\}.
\end{equation}
The middle third Cantor set corresponds to $\mathcal{C}_{3,\{ 0, 2 \}}$, that is, base $b=3$  
and allowed digits $D= \{0,2\}$.  
A key fact that we use is that the representation of a number in a Cantor set is uniquely determined by its digits.

\smallskip 
\noindent\textbf{Standing assumption.} 
Throughout the paper, with the single exception of \S\ref{sub:signed}, the base $b\geq 3$ and the digit set $D$ satisfy 
\begin{equation}\label{e:standing}
D\subset \{0,1,\ldots ,b-1\}, \qquad 0\in D, \qquad 2\leq \lvert D\rvert <b. 
\end{equation}
We write $\mathcal C = \mathcal C_{b,D}$, and list its elements in increasing order as $0=k_0<k_1<k_2<\cdots$. 

\begin{remark} 
The condition $0\in D$ is natural. It is necessary for $\mathcal C$ to be intersective (see \S\ref{sec;intersective}). 
It also makes the map from base-$\lvert D\rvert$ expansions of $n$ to $k_n$ an order preserving bijection (Proposition \ref{p;CantorStructure}), 
so that the first $\lvert D\rvert^\ell$ elements of $\mathcal C$ are exactly $\mathcal C\cap[0,b^\ell)$. 
Without it, the base-$\lvert D\rvert$ expansion of $n$ no longer indexes $\mathcal C$, since such expansions have variable length and leading zeros are not permitted.
For instance, for $b=3$ and $D=\{1,2\}$ we have $k_0=1$, $k_1=2$, $k_2=4$, while Proposition \ref{p;CantorStructure}(i) would predict $k_2=d_0+d_1b=7$.
In particular the block identity $k_{\lvert D\rvert^in+j}=b^ik_n+k_j$ of Proposition \ref{p;CantorStructure}(ii) fails, and several of our arguments do not apply. 
The case $\lvert D \rvert=1$ is excluded, since then $D=\{0\}$ and $\mathcal{C}=\{0\}$. 
\end{remark}

A more detailed summary of the results of this paper follows, with additional details in each section below.    
We establish simple sufficient conditions for a Cantor set $\mathcal C$ to be \emph{intersective} with power savings, that is, there are concrete constants $\delta >0$ and $C$ so that for all large $N$, every subset $A\subset \{1, \ldots, N\}$ with $\lvert A\rvert \geq CN^{1- \delta }$ contains two elements whose difference is a non-zero element of the Cantor set. 
In fact these sets $\mathcal C$ are \emph{van der Corput} as well, and can be of arbitrarily small dimension while still having power savings. 
These examples use digit sets which contain negative digits, and so are the one place where we depart from our standing assumption; see \S\ref{sub:signed}.  These are deterministic examples, which complement the recent result of Green \cite{green2023sarkozystheoremshiftedprimes} showing that the shifted primes also satisfy the stronger van der Corput property with power savings.

For a Cantor set of integers to be intersective, it is necessary that $0$ is an allowed digit. We show every Cantor set satisfying our standing assumption is van der Corput, and hence intersective. We record the elementary bound $\tau(\mathcal C,N)\leq \log b/\log N$ in Proposition \ref{p:logSavings}; whether these sets have a power savings we do not know, see Question \ref{q:powersavings}.  
We discuss the intersective property for certain polynomials evaluated along Cantor sets. 
That these sets are intersective already follows from the polynomial IP recurrence theorem of Bergelson, Furstenberg and McCutcheon \cite{MR1417769}, since a Cantor set contains an IP set. 
What is new is a density statement for the set of return times, measured along the natural enumeration of the Cantor set.  
These results are presented in \S \ref{sec;intersective}.  

The polynomial intersective result uses Furstenberg's correspondence principle, 
much as it first appeared in \cite{MR0498471}, and depends in part on the uniform distribution results for Cantor sets established by Coquet 
\cites{MR506123,MR578819}. We discuss uniform distribution modulo $1$ in \S\ref{sec:uniformDistributionMod1} and obtain extensions of those results to the modular case in \S\ref{sec:modular}, a theme that was the focus of \cite{ERDOS199899} and \cite{SAAVEDRA-ARAYA_2026}.   
The uniform distribution results imply certain mean ergodic theorems, recorded in \S\ref{s;meanErgodic}.  The question of pointwise convergence arises. For the linear averages along Cantor sets, it has recently been answered by Brokering Pinilla, Iosevich and Krause \cite{pinilla2026pointwiseconvergenceergodicaverages}, for $L^p$ functions with $2\leq p<\infty$; the polynomial case remains open, as far as we know. 

The last section \S\ref{sec;PairCorrelation} notes that Cantor sets whose digits are at most $(b-1)/2$ enjoy a property stronger than classical uniform distribution, namely \emph{metric pair correlation}, with a bound on the dimension of the exceptional set in terms of the additive energy of $D$. 
Extensions of this observation would be interesting to pursue.

\smallskip

 \subsection{Background and history}\label{sub:history}

Integer Cantor sets, also called ``restricted digit sets'', are a natural object of study.  The earliest result 
known to us is from 
1965 in Fine's paper \cite{MR184895}, which studied the distribution of the sum of digits function modulo $p$, a question closely related to Proposition \ref{p;sumDigits} below. 
Weyl's polynomial equidistribution theorem has been extended to Cantor sets by Coquet  \cites{MR506123,MR578819}. We will present an extension of these results below.  This was followed by work of Erd\H{o}s, Mauduit and S\'{a}rk\"{o}zy \cites{ERDOS199899,MR1692287} who considered the modular joint  distribution of Cantor sets of integers, with their   `sum of digits' function.  They established uniform distribution modulo $n$, with error rates on convergence.  
Their work is extended by 
Konyagin \cite{MR1832701}
and more recently 
Saavedra-Araya in \cite{SAAVEDRA-ARAYA_2026}, who used Markov chain techniques to generalize  such sets.  

The definition of `classical' Cantor sets has many possible extensions, with definitions in terms of systems of linear equations \cites{MR2678855,MR2290789}, 
automata, as for instance the Thue--Morse sequence \cite{MR1830572}, Markov Chains \cite{SAAVEDRA-ARAYA_2026},
and more general dynamical variants 
as studied by 
Glasscock, Moreira and Richter 
\cite{MR4734865}.

The role of digit functions and restricted digits along the prime integers was raised by Gelfond \cite{MR220693}, who asked for the distribution of the sum of digits function, for a given base $b$, mod $m$. This was answered 
by Mauduit and Rivat \cite{MR2680394}.
This result attracted a great deal of attention to these and related problems.

Particularly relevant is Maynard's seminal work \cite{MaynardInvent2019}, which proved the existence of infinitely many prime numbers in certain kinds of Cantor sets; the work of Leng and Sawhney \cite{leng2024vinogradovstheoremprimesrestricted}, which established a Goldbach-type result, extending Vinogradov's theorem to this setting; and the work of Nath \cite{jlms.12837}, who proved Bombieri--Vinogradov type theorems for primes with a missing digit. 
In a different direction Kra and Shalom \cite{kra2025ergodicaverageslargeintersection}  
have studied multiple recurrence for `rational spectra' IP sets.  Here IP can stand for   `infinite parallelepipeds'.  In our setting, any Cantor set with base $b$ and two allowed digits, one of which is zero, forms an IP set, with rational spectra.

\section{Preliminaries}

\subsection{Notation} 
We use standard Vinogradov notation, thus $A\ll B$ means that $\lvert A\rvert \leq C \lvert B\rvert$, 
for an absolute constant $C$. By $A=B+O(\delta)$ we mean that $A-B \ll \delta$.  

For integers $N$, we set $[N] = \{1,2, \ldots, N\}$. 
For a finite set $A$ and $f \colon A\to \mathbb C$, we write 
\begin{equation}
    \Expectation_A f = \frac{1}{\lvert A\rvert} 
    \sum_{n\in A} f(n). 
\end{equation}
And we will frequently write $\Expectation_{[N]}=\Expectation_N$.  
We index the elements of a Cantor set from $n=0$, and averages over the index $n$ are written $\Expectation_{0\leq n<N}$, or $\Expectation_{M\leq n<N}$; 
the distinction from $\Expectation_{[N]}$ matters only in that $k_0=0$ is included.  
When integrating over a probability space $(X,P)$, we will write $\Expectation_X$. 
Exponentials are denoted by $e(x)=e^{2 \pi i x}$. 

\subsection{Measure Preserving Systems} 
We recall standard notions around measure preserving systems.  
In particular, $(X,\mathcal A,\mu,T)$ 
is a \emph{measure preserving system} if $(X,\mathcal A,\mu)$ is a probability space, and $T \colon X\to X$ preserves $\mu$ measure, in that $\mu(T^{-1} A)=\mu(A)$ for all measurable $A\in \mathcal A$.

\subsection{Cantor Sets} % (fold)
\label{sub:cantor_sets}

We collect here important, yet elementary, properties of the Cantor sets we work with. 
Below, the elements of a Cantor set $\mathcal C$ are written as $k_0<k_1 < \cdots$, and the \emph{sum of digits function} $s_b(k)$ is done relative to the base $b$ of the Cantor set. Thus, writing $k=\sum_{c=0}^t k_c b^c$, with $k_c \in \{0,...,b-1\}$, we have $s_b(k)= \sum_{c=0}^t k_c $.  
Since $D\subset\{0,\ldots,b-1\}$, for $k=\sum_j d_jb^j\in\mathcal C$ we have $s_b(k)=\sum_j d_j$. 

We also let $s$ denote the largest integer such that $D$ is contained in an arithmetic progression of step size $s$. 
Since $0\in D$, we have $s=\gcd(D)$, and in particular $s\mid d$ for all $d\in D$, and $\mathcal C\subset s\mathbb{Z}$.

% \textbf{color}{blue}{Note:Digit set is required be in the 'natural residues'.} 

\begin{proposition} \label{p;CantorStructure}
    Let $\mathcal C=\mathcal C_{b,D}=(k_n)_{n\geq 0}$, and write $D=\{0=d_0<d_1<\cdots<d_{|D|-1}\}$.  Then, the following are true: 
    \begin{enumerate}
        \item[(i)] For any $n\in \mathbb{N}_0$, written in base-$|D|$ as $n=\sum_{c=0}^tn_c|D|^c$ with $n_c\in \{0,...,|D|-1\}$, then $k_n=\sum_{c=0}^td_{n_c}b^c$.  And $s_b(k_n)=\sum_{c=0}^td_{n_c}$.  
        %Note: we do not allow leading zeroes in the base-$|D|$ representation of $n$ unless $0\not \in D$. \\
        \item[(ii)] For any $i\geq 1 $, $n \geq 0$, and $0\leq j < |D|^i$, $k_{|D|^in+j}=b^ik_n+k_j$, 
        and $s_b(k_{|D|^in+j})=s_b(k_n)+s_b(k_j)$. 
    \end{enumerate}
\end{proposition}

\begin{proof}
    With $0$ as an allowed digit, the point (i) is clear, because the map $\phi:\mathbb{N}_0\rightarrow \mathcal C$ given by  \begin{align*} n&\mapsto \sum_{c=0}^td_{n_c}b^c\end{align*} 
    is both bijective and monotone (with respect to increasing orderings on $\mathbb{N}_0$ and $\mathcal C$). \\
    
    For the point (ii), write $n=\sum_{c=0}^tn_c|D|^c$ and $j=\sum_{c=0}^{i-1}j_c|D|^c$ with $n_c,j_c\in \{0,...,|D|-1\}$. Then, $k_n=\sum_{c=0}^td_{n_c}b^c$ and $k_j=\sum_{c=0}^{i-1}d_{j_c}b^c$. Note that since the most significant digit of $j$ is at most in the $(i-1)$-st place, we may write $|D|^in+j=\sum_{c=i}^{t+i}n_{c-i}|D|^c+\sum_{c=0}^{i-1}j_c|D|^c$, so that, by part (i) $k_{|D|^in+j}=\sum_{c=i}^{t+i}d_{n_{c-i}}b^c+\sum_{c=0}^{i-1}d_{j_c}b^c$. By factoring $b^i$ out of the first sum and reindexing, the result follows.

\end{proof}

% subsection cantor_sets (end)

% Section 
\section{Intersective and van der Corput Properties } \label{sec;intersective}
We turn to more explicitly arithmetic properties of these Cantor sets. Namely, the properties of being \emph{intersective} or \emph{van der Corput} sets.  

\begin{definition}
    We say that a set of integers $H \subset \mathbb N $ is \emph{intersective} if for 
 \begin{equation} \label{eSmax}
     I (H,N) = \max 
     \{ \lvert A\rvert /N \colon A\subset  [N], \ (A-A) \cap (H\setminus\{0\}) = \emptyset\}, 
 \end{equation}
 we have $I (H,N) \to 0$ as $N\to \infty$.   
 (The element $0$ must be excluded, since $0\in A-A$ for every non-empty $A$, and our Cantor sets contain $0$.) 
 We say that $H$ has a \emph{power savings (of order $c$)} if $I (H,N)  \ll N ^{-c}$.
 \end{definition}

It is the Theorem of S\'ark\"ozy \cite{MR487031} and Furstenberg \cite{MR0498471} that the square integers are intersective. 
S\'ark\"ozy also showed that the primes, plus or minus $1$, are also intersective. 
Both results have been the subject of a great deal of interest.  The best current result on the square integers is due to Green and Sawhney \cite{green2025newboundsfurstenbergsarkozytheorem}.

We recall the definition of \emph{van der Corput set} of Kamae and Mend\`es-France
\cite{MR0516154}.  We give a quantitative definition  of this condition here. 

\begin{definition} 
For a set of integers $H \subset \mathbb{N}$  and $N\in \mathbb{N}$, set  
\begin{equation}\label{e;THN}
    \tau(H,N) = \inf _{\Psi \in \mathcal{T}_{N,H}} a_0
\end{equation}
where $\mathcal{T}_{N,H}$ is the set of all trigonometric polynomials $\Psi(x)$ satisfying 
\begin{equation} \label{e:T}
    \Psi(x) = a_0 + \sum_{ n<N, \ n\in H\setminus\{0\}} a_n \cos(2\pi n x) , 
\end{equation}
with $a_n\in \mathbb R$, $\Psi(0)=1$ and $\Psi(x)\geq 0$ for all $x$. 
We reserve $T$ for a measure preserving transformation, and use $\tau$ for the quantity \eqref{e;THN}. 
The sequence $H$ is said to be \emph{van der Corput (vdC)} if $\tau(H,N) \to 0$. 
That this agrees with the original definition of Kamae and Mend\`es-France, recalled below, is due to Ruzsa \cite{Ruzsa1984}. 
It is said to have a \emph{power savings of order $c$} in the van der Corput property if $\tau(H,N) \ll N ^{-c}$.  
\end{definition}

We recall some facts about the van der Corput property from \cite{MR0516154}: 
\begin{itemize}
    \item  $H$ is vdC if and only if it has the following property: for every sequence $\{x_n\} \subset [0,1)$, 
    if the sequences $\{x_{n+h}-x_n \colon n\in \mathbb N\}$ are uniformly distributed (mod $1$) for all $h\in H$, then $\{x_n\}$ is uniformly distributed (mod $1$). 
    This is the original definition of Kamae and Mend\`es-France. 

    \item A sufficient condition for $H$ to be vdC is that for every $n$ there are integers  $a_1 < a_2 < \cdots < a_n$ 
    with $a_j - a_i \in H$ for all $1\leq i < j \leq n$.  
    
    \item Kamae and Mend\`es-France \cite{MR0516154} showed that the van der Corput property for $H$ implies that $H$ is intersective. 
    In the quantitative form used here, $I(H,N) \ll \tau(H,N)$; see Ruzsa \cite{Ruzsa1984}. 
    The reverse implication does not hold, as was shown by Bourgain \cite{Bourgain1987}.  
\end{itemize}

We recall the definition of an IP-set, as defined by Furstenberg and Katznelson \cite{MR833409}: these are sets of the form $\{n_{t_1}+ \cdots +n_{t_k}\colon k\geq 1,\ t_1<t_2<\cdots<t_k\}$, that is, all possible finite sums of distinct elements of some infinite sequence $\{n_t\}$ of positive integers.
For a general digit set $D$, the Cantor set $\mathcal C_{b,D}$ contains an IP-set if and only if $0$ is an allowed digit; this is one reason for our standing assumption \eqref{e:standing}. 
Indeed, if $0\in D$, then for any $d\in D\setminus \{0\}$ the collection of sums $\sum_{t\in F}db^t$, over finite non-empty $F\subset \mathbb{N}_0$, is an IP-set contained in $\mathcal C$. 
Conversely, an IP-set generated by $\{n_t\}$ contains a non-zero multiple of every integer $q\geq 1$: among the partial sums $S_0=0$ and $S_v=n_1+\cdots +n_v$, two of $S_0,\ldots ,S_q$ agree modulo $q$, say $S_v\equiv S_{v'}$ with $v<v'$, and then $n_{v+1}+\cdots +n_{v'}$ lies in the IP-set and is divisible by $q$. 
Taking $q=b$, a non-zero element of $\mathcal C$ divisible by $b$ must have $0$ as its last digit. 
Under the standing assumption, $\mathcal C$ is vdC, since $\mathcal C$ contains an infinite sequence $u_0 < u_1 < \cdots$ with $u_j-u_i\in \mathcal C$, for $i<j$. 
Indeed, fix $d\in D\setminus\{0\}$ and set 
\begin{equation} \label{e:uj}
    u_j=\sum_{t<j}db^t , \qquad j\geq 0. 
\end{equation}
Thus, $\mathcal C$ is intersective. 
Making this argument quantitative is immediate, and gives a bound of order $1/\log N$.

\begin{proposition}\label{p:logSavings}
Let $\mathcal C=\mathcal C_{b,D}$ satisfy the standing assumption \eqref{e:standing}. Then, for all integers $N\geq 2$, 
\begin{equation}\label{e:logSavings}
\tau(\mathcal C, N) \leq \frac{\log b}{\log N} , \qquad\textup{and consequently}\qquad I(\mathcal C,N) \ll \frac{\log b}{\log N}. 
\end{equation}
\end{proposition}

\begin{proof}
Let $n\geq 1$ be an integer, and let $u_0<u_1<\cdots <u_{n-1}$ be as in \eqref{e:uj}, so that $u_0=0$. 
For $i<j$ we have $u_j-u_i=\sum_{i\leq t<j}db^t\in\mathcal C\setminus\{0\}$, and distinct pairs $i<j$ give distinct digit strings, hence distinct differences. 
Also, $u_{n-1}=d\frac{b^{n-1}-1}{b-1}\leq b^{n-1}-1$, so all these differences are less than $N_0\coloneqq b^{n-1}$. 
Consider the Fej\'er type kernel 
\begin{align}
    P (x) = \Bigl\lvert \Expectation_{0\leq j<n} e(u_jx)\Bigr\rvert^2 
    = \frac {1}{n ^2 } \sum_{i,j=0} ^{n-1} e\bigl( (u_j-u_i)x\bigr) . 
\end{align}
By construction $P\geq 0$ and $P(0)=1$. The frequencies $u_j-u_i$ and $u_i-u_j$ occur with the same coefficient $n^{-2}$, so $P$ is a cosine polynomial with frequencies in $\mathcal C\cap [0,N_0)$. 
Only the $n$ diagonal terms contribute to the constant term, so $a_0=n^{-1}$ and $\tau(\mathcal C,N_0)\leq n ^{-1} $. 
Given $N\geq 2$, take $n=1+\lfloor \log N/\log b\rfloor$, so that $N_0=b^{n-1}\leq N$ and $n\geq \log N/\log b$. 
As $N\mapsto \tau(\mathcal C,N)$ is nonincreasing, $\tau(\mathcal C,N)\leq \tau(\mathcal C,N_0)\leq n^{-1}\leq \log b/\log N$. 
The bound on $I(\mathcal C,N)$ follows from $I(\mathcal C,N)\ll \tau(\mathcal C,N)$. 
\end{proof}

We do not know whether every such Cantor set has power savings. 

\begin{question}\label{q:powersavings}
Let $\mathcal C=\mathcal C_{b,D}$ satisfy the standing assumption \eqref{e:standing}. 
Can \eqref{e:logSavings} be improved to a power savings, that is, is $\tau(\mathcal C,N)\ll N^{-c}$, or $I(\mathcal C,N)\ll N^{-c}$, for some $c=c(b,D)>0$? 
This is open to us even for the middle third Cantor set $\mathcal C_{3,\{0,2\}}$. 
Short of that, it would be of interest to improve the dependence on $b$ in \eqref{e:logSavings}. 
\end{question}
Power savings for Cantor sets with signed digits are established in \S\ref{sub:signed}. 
A necessary condition for a set of integers $H$ to be intersective is that $H$ contains a non-zero multiple of $q$, for every integer $q\geq 1$. 
(Otherwise, the sets $A=q\mathbb{Z}\cap[N]$ have $(A-A)\cap (H\setminus\{0\})=\emptyset$ and $\lvert A\rvert\geq N/q-1$.) 
In particular, if $0\notin D$, then $\mathcal C_{b,D}$ contains no multiple of $b$, and so is not intersective.

\subsection{Power savings for signed digit sets} \label{sub:signed}
Green \cite{green2023sarkozystheoremshiftedprimes} has shown that the shifted primes have a power savings in the van der Corput property.  
We see below that a natural condition on the Cantor set implies that the  integer Cantor sets 
have a power savings. 

\emph{In this subsection only, we depart from the standing assumption \eqref{e:standing}.} 
Let $R_b\subset (-b,b)$ be a complete system of residues modulo $b$ with $0\in R_b$, and let $D\subset R_b$ with $0\in D$ and $2\leq \lvert D\rvert<b$. 
The digits in $D$ are allowed to be negative, but we insist that $\lvert d\rvert\leq b-1$ for $d\in D$. 
This normalization is used below to bound the elements of $\mathcal C_{b,D}$ in terms of the number of their digits: 
if $k\in\mathcal C_{b,D}$ has digits in the places $0,\ldots ,J$, then $\lvert k\rvert \leq (b-1)\sum_{j=0}^{J} b ^{j} <b ^{J+1} $. 
We define $\mathcal C_{b,D}$ by the same formula as before; the representation of an element of $\mathcal C_{b,D}$ by its digits is again unique. 
If $D$ contains negative digits, $\mathcal C_{b,D}$ contains negative integers. 
Since the intersective and van der Corput properties only concern the set of differences, which is symmetric, 
for such sets we apply the definitions above to $H=\{\lvert k\rvert \colon k\in \mathcal C_{b,D}\}\setminus\{0\}$. 
For instance, one can take $b=5$, $R_5=\{-2,-1,0,1,2\}$ and $D=\{-1,0,1\}$. 

\begin{proposition}\label{Powersavings}  
Let $\mathcal{C} = \mathcal{C}_{b,D}$ be a Cantor set with base $b$, and signed digit set $D\subset R_b$ as above.  
If $S$ is a finite set of integers with $S-S \subseteq D$ and $\lvert S\rvert >1$, then 
$\mathcal{C}$ is  van der Corput (hence intersective) with  power savings with $c =  \frac{ \log \lvert S\rvert }{\log b} $. 
    \end{proposition}

These examples supplement the primes minus one as concrete examples of sets of integers with a power savings in the intersective property, while also being `sparse', in that their density also decreases with a power savings.  

\begin{proof}
 Let $\sigma=\lvert S\rvert$, and note that $0\in S-S$ is an allowed digit. 
 For an integer $J\geq 0$, consider the product of Fej\'er type kernels 
 \begin{align}
     P_J (x) = \prod _{j=0}^J  \Bigl\lvert \Expectation_{s\in S} e(sb^jx)\Bigr\rvert^2 . 
 \end{align}
By construction, $P_J(0)=1$ and $P_J\geq 0$. 
Expanding each factor as $\sigma^{-2}  \sum_{s,t\in S} e((s-t)b^jx)$, we have 
\begin{align}
    P_J (x) = \sigma^{-2(J+1)} \sum_{ (s_j),(t_j)\in S^{J+1} } e\Bigl( 
    x\sum_{j=0}^J (s_j-t_j) b^j  \Bigr). 
\end{align}
Each frequency $n=\sum_{j=0}^J (s_j-t_j) b^j$ has digits $s_j-t_j\in S-S\subseteq D$, so $n\in\mathcal C$, and $\lvert n\rvert<b^{J+1}$, since $s_j-t_j\in D\subset R_b\subset (-b,b)$. 
As $S-S$ is symmetric, the frequencies $n$ and $-n$ appear with the same coefficient, so $P_J$ is a cosine polynomial, with real coefficients and frequencies in 
$\{\lvert k\rvert\colon k\in\mathcal C\}\cap[0,N)$, where $N=b^{J+1}$. 

It remains to compute the constant term, which is 
\begin{align}
    a_0 = \sigma^{-2(J+1)} \Bigl\lvert \Bigl\{  (s_j),(t_j)\in S^{J+1} \colon  \sum_{j=0}^J (s_j-t_j) b^j =0 \Bigr\}\Bigr\rvert. 
\end{align}
Since $s_j-t_j\in D\subset R_b$, and the representation of an element of $\mathcal C$ by its digits is unique, 
$\sum_{j=0}^J (s_j-t_j) b^j =0$ holds only if $s_j=t_j$ for all $j$. 
So only the $\sigma^{J+1}$ diagonal terms contribute, and 
\begin{align}
    a_0 = \sigma^{-2(J+1) + (J+1)} =\sigma^{-(J+1)}= N^{-\log\sigma/\log b}.   
\end{align}
This shows that $\tau(\mathcal C,N) \leq  N^{-\log\sigma/\log b}$ along the sequence  $N =b^{J+1}$, $J\geq0$. 
The function $N\mapsto \tau(\mathcal C,N)$ is nonincreasing, hence $\tau(\mathcal C,N) \leq b^{\log\sigma/\log b}  N^{-\log\sigma/\log b}=\sigma N^{-\log\sigma/\log b}$ 
for all integers $N\geq b$. 
\end{proof}

The above proposition is an instance of the following general principle. 
Let $E$ be a finite set of integers, and suppose that every non-zero difference of elements of $E$ lies in $\pm K$, for a set of positive integers $K$. 
If $A\subset[N]$ satisfies $(A-A)\cap K=\emptyset$, then the translates $A+e$, $e\in E$, are pairwise disjoint, and contained in an interval of length $N+\max E-\min E$. 
Hence $\lvert A\rvert\leq (N+\max E-\min E)/\lvert E\rvert$. 
In the Proposition, one takes $E=\{\sum_{j=0}^J s_jb^j\colon s_j\in S\}$; these sums are distinct, by uniqueness of the digit representation, so that 
$\lvert E\rvert=\lvert S\rvert^{J+1}$, and $\max E-\min E<b^{J+1}$ since the differences have digits in $D\subset (-b,b)$. 
Proposition \ref{p:logSavings} is the instance of the same principle in which $E=\{u_0,\ldots ,u_{n-1}\}$, with $u_j$ as in \eqref{e:uj}. 
In general, difference sets have a strong vdC property; see Example 2 in \cite{MR0516154}.

Here, we note simple sufficient conditions for $D$ to contain a difference set.

\begin{proposition} \label{p:signedSufficient}
  Let $\mathcal C=\mathcal C_{b,D}$ be a Cantor set with signed digit set $D\subset R_b$ as above, and suppose either 
  \begin{enumerate}
      \item $D$ contains $d$ and $-d$ as integers, for some $d\neq 0$. 
      \item The residue system $R_b$ is symmetric ($R_b= \{-\frac {b-1}2,\ldots ,\frac {b-1}2\}$ for $b$ odd, 
      and $R_b=\{-b/2+1,\ldots b/2\}$ for $b$ even) and $\lvert D\rvert \geq  \lfloor b/2\rfloor+2$. 
  \end{enumerate}
Then, $\mathcal C$ is van der Corput with power savings at least $\log 2/\log b$. 

%     Suppose that $|D|> \frac{b+1}{2}$, or that $D$ contains a pair of antipodal points $d,-d\in \mathbb{Z}_b$. Then the Cantor set $C(b,D) $ is van der Corput with  power-savings of at least $\frac{\log 2}{\log b}$. 
\end{proposition} 

\begin{proof}
 If the first condition holds, the van der Corput property with power savings follows from Proposition \ref{Powersavings}. 
 And, the second condition implies the first: For $b$ odd, the residue system $R_b$ is partitioned into $0$ and $\frac {b-1}2$ 
 symmetric residue systems of the form $\{d,-d\}$ for $d=1,\ldots ,\frac {b-1}2$. That is a total of $\frac{b+1}2= \lfloor b/2\rfloor +1 $ residue systems.  
 And if $b$ is even, the residues $R_b \setminus \{0,b/2\}$, are partitioned into $\frac{b-2}2 = b/2-1$ symmetric residue systems. 
 In both cases the cardinality condition $\lvert D\rvert \geq  \lfloor b/2\rfloor+2$ implies that $D$ must contain a symmetric pair  of non-zero residues.

%    Observe that if $|D|> \frac{b+1}{2}$, then $D$ contains a pair of nonzero antipodal points, so it suffices to prove that the claim follows from the second point. To do this, simply choose $S=\{0,d\}\subset \mathbb{Z}_b$ and apply Proposition \ref{Powersavings}.
\end{proof}

Next, we can give examples that relate the dimension of the Cantor set to the power savings.  In particular, there are small dimension Cantor sets that have power savings. 
(The shifted primes are full dimensional.)

\begin{proposition} \label{p:dimension}
    For each  $\theta \in (0,1)$ and $ 0<\epsilon<\min\{\theta ,1-\theta \}$, there exists an integer Cantor set $\mathcal{C}=\mathcal{C}_{b,D}$, with a signed digit set $D$ as above, for which the following two points hold simultaneously: 
    \begin{enumerate}
        \item[(1)] $  \theta -\epsilon \leq \dim(\mathcal{C})=\frac{\log |D|}{\log b}  \leq \theta + \epsilon$, where $\dim(\mathcal C)=\lim_{N\to\infty}\frac{\log\lvert\mathcal C\cap[-N,N]\rvert}{\log N}$. 
        \item[(2)] There is a power savings of at least $\theta -\epsilon$ 
        for $\mathcal{C}$. 
    \end{enumerate}
\end{proposition}

\begin{proof}
First choose an odd base $b$ so large that 
\begin{equation} \label{e:bLarge}
    \frac{\log 5}{\log b}<\frac \epsilon 2 \qquad\textup{and}\qquad b ^{1-\epsilon /2}\leq b-1 , 
\end{equation}
and equip it with the symmetric residue system $R_b=\{-\frac{b-1}2,\ldots,\frac{b-1}2\}$. 
As $k$ ranges over $0,1,2,\ldots$, the quantity $\frac{\log (4k+1)}{\log b}$ increases from $0$ to $\infty$ in steps of at most $\frac{\log 5}{\log b}<\frac \epsilon2$. 
So we may fix $k\geq 1$ with $\frac{\log(4k+1)}{\log b}\in (\theta -\frac \epsilon 2,\ \theta +\frac \epsilon 2)$. 
Since $\epsilon <1-\theta $, we have $\theta +\frac \epsilon 2\leq 1-\frac \epsilon 2$, and so $4k+1<b ^{1-\epsilon /2}\leq b-1$ by \eqref{e:bLarge}. 
Therefore $D:=\{-2k,-2k+1,\ldots ,0,\ldots ,2k-1,2k\}$ is contained in $R_b$, and $\dim (\mathcal{C})=\frac{\log (4k+1)}{\log b}$, which is the point (1). 

For (2), put $S:=\{-k,\ldots ,0,\ldots ,k\}$, so that $S-S\subset D$. 
Proposition \ref{Powersavings} applies, and gives a power savings of 
\begin{equation}
\frac{\log (2k+1)}{\log b}=\dim \mathcal{C}-\frac{1}{\log b}\log \frac{4k+1}{2k+1} 
> \dim \mathcal{C}-\frac{\log 2}{\log b} > \Bigl(\theta -\frac \epsilon 2\Bigr)-\frac \epsilon 2 = \theta -\epsilon . 
\end{equation}
\end{proof}

\begin{remark}
The power savings of a set $H$ cannot exceed its dimension. 
Indeed, for any set $H$ of integers, choose $A\subset[N]$ greedily: run through $x=1,\ldots,N$, and add $x$ to $A$ if $x\notin a\pm (H\setminus\{0\})$ for every $a$ already in $A$. 
Each element of $A$ excludes at most $2\lvert H\cap[N]\rvert$ further elements, so $(A-A)\cap (H\setminus\{0\})=\emptyset$ and 
$I(H,N)\geq (2\lvert H\cap[N]\rvert+1)^{-1}$. 
For $H=\{\lvert k\rvert\colon k\in\mathcal C\}$ we have $\lvert H\cap[N]\rvert=N^{\dim(\mathcal C)+o(1)}$, so a power savings of order $c$ forces $c\leq\dim(\mathcal C)$. 
Since $I(H,N)\ll \tau(H,N)$, the same holds for the van der Corput property. 
Thus Proposition \ref{p:dimension} is sharp, up to $\epsilon$. 
\end{remark}

\subsection{Polynomial recurrence along Cantor sets}
From now on, we return to the standing assumption \eqref{e:standing}. 
We show that the intersective property continues to hold for certain polynomials evaluated along a Cantor set.  
The proof of Theorem \ref{t:polysIntersective} below draws on \S\ref{sec:uniformDistributionMod1} and \S\ref{sec:modular}, 
specifically on Theorem \ref{t:polyUniform}, Lemma \ref{l:Allq} and Lemma \ref{lem:positiveDensity}. 
Those sections do not depend on this one, so there is no circularity.

\begin{theorem} \label{t:polysIntersective}
 Let $\mathcal C = \mathcal C_{b,D} = \{0=k_0<k_1< \cdots \}$, and $p $ be  a  polynomial 
 mapping the integers to the integers, with  positive leading coefficient 
 \emph{and no constant term.}
Set 
\begin{equation} \label{e:Hp}
H_p \coloneqq \{ \lvert p(k)\rvert \colon k \in \mathcal C\}\setminus \{0\} \subset \mathbb{N}. 
\end{equation}
Then, $H_p $ is intersective.  
Moreover, for any set of integers $A$ of positive upper Banach density $d^\ast(A)>0$, 
\begin{equation} \label{e;CantorDensity}
    \liminf_{N\to\infty} \Expectation_{0\leq n<N}
     d^\ast (A\cap (A - p(k_n)))  >0. 
\end{equation}
Consequently, there is a $c = c(A,p,\mathcal C)>0$ so that  the set of integers $\{n \colon d^\ast(A\cap (A-p(k_n))) >c \}$ 
has lower density at least $c$. 
\end{theorem}

The passage to absolute values in \eqref{e:Hp} is a convenience. 
Our definition of intersective is stated for subsets of $\mathbb{N}$, while a polynomial with positive leading coefficient and no constant term need be neither non-negative nor injective on $\mathcal C$: 
$p(x)=x^2-5x$ is negative at $x=1$, and $p(x)=x^2-x$ vanishes at $x=1$. 
Since the intersective property only concerns the symmetric set $A-A$, replacing $p(\mathcal C)$ by $H_p$ loses nothing.

A Cantor set contains an IP set, hence the recurrence property, namely that $d^\ast(A\cap A-p(k))>0$ for some non-zero $k\in\mathcal C$, is an immediate consequence of the polynomial IP recurrence theorem of Bergelson, Furstenberg and McCutcheon \cite{MR1417769}. 
Their result, and its multiple polynomial recurrence extension by Bergelson and McCutcheon \cite{BergelsonMcCutcheon2000}, show that the set of return times is large in the sense of IP sets: it is an IP$^\ast$ set. 
As far as we are aware, these results do not give information about the density of the return times measured along the enumeration $k_0<k_1<\cdots$ of $\mathcal C$. 
An IP$^\ast$ set must meet every IP set contained in $\mathcal C$, but this does not by itself imply that it contains a positive proportion of the $k_n$. 

The new information is the density result \eqref{e;CantorDensity}: The set of $n$ for which one has recurrence along $p(k_n)$ has positive lower density.  

For the proof, we rely upon the Furstenberg correspondence, the structure of dynamical systems, and  the following property of Cantor sets detailed in 
Lemma \ref{l:Allq}.  
In particular, that Lemma proves that for Cantor sets 
 $\mathcal C = \mathcal C_{b,D} 
= \{k_0<k_1 < \cdots\}$, 
the sequence $ k_n\mod q$ converges to a distribution on $\mathbb{Z} _q$, for all $q$, which assigns positive mass to $0$ (Lemma \ref{lem:positiveDensity}).  But   the limiting measures   are not always uniform.

\begin{proof}
[Proof of Theorem \ref{t:polysIntersective}]
We use the Furstenberg correspondence principle \cite{MR0498471}. 
This gives us an \emph{invertible} system $(X,\mathcal B, \mu, T)$ --- it is the shift on a closed shift-invariant subset of $\{0,1\}^{\mathbb{Z}}$ --- 
and $\tilde A\subset X$ with $\mu(\tilde A)=d^\ast (A)$ for which 
$d^\ast(A\cap (A-m)) \geq \mu (\tilde A\cap T^{-m} \tilde A)$ for every integer $m$. 
And so, it suffices to show that 
\begin{equation} \label{eENA}
\liminf_N \Expectation _{0\leq n<N} \int_X \mathbf{1}_{\tilde A} T^{p(k_n)}  \mathbf{1}_{\tilde A} \;d\mu > 0. 
\end{equation}
In so doing, we will appeal to the fact that all terms are non-negative. In particular, we can always pass to averages over sets of $n$ of positive density.

Split the Hilbert space $L^2(X)$ into $ \mathcal H _{\textup{rat}} \oplus \mathcal H _{\textup{0}}$, where 
\begin{equation}
 \mathcal H _{\textup{rat}}
 = \textup{clos}\{f\in L^2(X)\colon  
 f = T^k f, \textup{for some integer $k\geq 1$}   \}, 
 \end{equation}
and $\mathcal H _{\textup{0}}$ is the closed orthogonal space.  
These spaces are invariant under action of $T$.

Let $\mathbf{1}_{\tilde A} = f_r + f_0$ be the corresponding decomposition of $\mathbf{1}_{\tilde A}$. Now $f_r \in \mathcal H _{\textup{rat}}$, and 
constants are in  $\mathcal H _{\textup{rat}}$, hence  $\lVert f_r\rVert_2^2 \geq  \mu(\tilde A)^2$. 
Moreover, for any integer $n$, we have 
\begin{equation}
    \int_X \mathbf{1}_{\tilde A}  T^n \mathbf{1}_{\tilde A}  \; d\mu 
    =
    \int_X f_r T^n f_r \;d\mu 
    +\int_X f_0 T^n f_0 \;d\mu . 
\end{equation}

Thus, given $0< \epsilon < \mu(\tilde A) ^2/5 $,  we can choose $h\in \mathcal H _{\textup{rat}}$,  so that 
$\lVert f_r - h \rVert_2 < \epsilon $. 
And also there is an integer $t$ so that  $h = T^t h$. 
In particular, it follows that $\int_X h \cdot T^{jt} h \; d\mu =\lVert h\rVert_2^2\geq 
\mu(\tilde A)^2- 2\epsilon $ for any $j$. 
Moreover, as $\lVert f_r\rVert_2\leq 1$, we have $\bigl\lvert\int_X f_r\cdot T^mf_r\;d\mu-\int_X h\cdot T^mh\;d\mu\bigr\rvert\leq 3\epsilon$ for every integer $m$.

We first address the limit obtained by replacing the indicator by $h$.  
The polynomial $p$ is integer valued, so its coefficients are rational; let $L\geq1$ be a common denominator of them, and set $t'=tL$. 
Let $B$ be the set of integers $n$ 
such that  $k_n\equiv 0 \mod t'$.  
This set has positive density by Lemma \ref{lem:positiveDensity}. 
Since the polynomial has no constant term, writing $p(x)=\sum_{i\geq1}c_ix^i$ with $Lc_i\in\mathbb{Z}$, for $k=t'v$ we have 
$c_ik^i=(Lc_i)\,t^i L^{i-1}v^i\in t\mathbb{Z}$. 
Hence $p(k_n)\equiv 0\mod t$ for all $n\in B$, and   
 $h = T ^{p(k_n)}h$.  
(Note that $t\mid k_n$ alone would not suffice: for $p(x)=x(x+1)/2$ and $t=2$, we have $p(2)=3$.) 
That is, only forming the average over the set $B$, we have 
\begin{align} \label{e:relPrime}
\liminf _{N} \Expectation _{\substack{ n\in B\cap  [N] \\  }} \int_X 
f_r \cdot T ^{p(k_n)} f_r \;d\mu  
\geq   \mu(\tilde A) ^2 - 5\epsilon . 
\end{align}
To conclude this case, it remains to argue that the function $f_0$ satisfies 
\begin{align} 
\lim _{N} \Expectation _{\substack{ n\in B\cap  [N] \\  }} \int_X 
f_0 \cdot T ^{p(k_n)} f_0 \;d\mu  
=0. 
\end{align}
And, since $B$ has positive density, it suffices to show that 
\begin{align} 
\lim _{N} \Expectation _{\substack{ n\in [N] \\  }} 
\mathbf{1}_{\{ k_n\equiv 0 \mod t'\} }
\int_X 
f_0 \cdot T ^{p(k_n)} f_0 \;d\mu  
=0. 
\end{align}
Since $\mathbf{1}_{ k_n\equiv 0 \mod t'}=\Expectation_{0\leq a<t'}e(\frac{a}{t'}k_n)$, 
it suffices to show that for each $0\leq a < t'$, 
\begin{align}  \label{eExpPoy}
\lim _{N} \Expectation _{\substack{ n\in  [N] \\  }} \int_X e\left(\frac{a}{t'}k_n \right) 
f_0 \cdot T ^{p(k_n)} f_0 \;d\mu  
=0.
\end{align}
This is a variant of our  polynomial Theorem \ref{t:polyUniform}.  
We address it via spectral methods.   
We know that 
\begin{align} 
\lim _{N} \Expectation _{\substack{ n\in  [N] \\  }} e\left(\frac{a}{t'} k_n  + \alpha p(k_n) \right) =0
\end{align}
for all irrational $\alpha$, by Theorem \ref{t:polyUniform}, since $p$ is non-constant.  
Letting $ \nu   $ be the Herglotz measure for the sequence $\int_X f_0 T^n f_0\;d\mu$, the expectation in \eqref{eExpPoy} is then 
\begin{equation}
   \int_{\mathbf{T}}\Expectation _{\substack{ n\in  [N] \\  }}  e\left(\frac{a}{t'} k_n  + \alpha p(k_n) \right)  \;d \nu (\alpha)  
\end{equation}
As $T$ is invertible, its Koopman operator is unitary, and $\nu$ is the spectral measure of $f_0$. 
An atom of $\nu$ at a rational $a/k$ would produce a non-zero $g$ with $Tg=e(a/k)g$, lying in the closed span of $\{T^nf_0\colon n\in \mathbb{Z}\}\subset \mathcal H_{\textup{0}}$. 
But then $g=T^kg$, so $g\in \mathcal H_{\textup{rat}}$, a contradiction. 
Thus $\nu$ has no rational point masses, and as $\mathbb{Q}/\mathbb{Z}$ is countable, $\nu(\mathbb{Q}/\mathbb{Z})=0$. 
Hence, by dominated convergence, the limit above is zero.  

Combining these, the averages of $\int_X \mathbf{1}_{\tilde A}\cdot T^{p(k_n)}\mathbf{1}_{\tilde A}\;d\mu$ over $n\in B\cap[N]$ have lower limit at least $\mu(\tilde A)^2-5\epsilon>0$. 
Since these integrals are non-negative, and $B$ has positive density $\pi$, the lower limit in \eqref{eENA} is at least $\pi(\mu(\tilde A)^2-5\epsilon)>0$. 
This proves \eqref{e;CantorDensity}, and the final assertion of the Theorem follows since the quantities $d^\ast(A\cap (A-p(k_n)))$ lie in $[0,1]$. 

It remains to see that $H_p$ is intersective. 
Observe first that the argument above establishes \eqref{eENA} for an \emph{arbitrary} invertible measure preserving system $(X,\mathcal B,\mu ,T)$ and any set $\tilde A$ with $\mu (\tilde A)>0$; 
the correspondence principle was used only to produce such a system. 
Observe also that $Z=\{n\colon p(k_n)=0\}$ has at most $\deg p$ elements, as $p$ is non-constant, and that $\lvert p(k_n)\rvert \in H_p$ for every $n\notin Z$. 

Suppose that $I(H_p,N)\not\to 0$. 
The quantity $\Lambda (N)=\max \{\lvert A\rvert \colon A\subset [N], \ (A-A)\cap H_p=\emptyset \}$ is subadditive, since for $A\subset [M+N]$ the traces of $A$ on $[M]$ and on $(M,M+N]$ are each admissible. 
Hence $I(H_p,N)=\Lambda (N)/N$ converges, necessarily to some $\lambda >0$. 
For each $N$, fix $A_N\subset [N]$ with $\lvert A_N\rvert \geq \lambda N$ and $(A_N-A_N)\cap H_p=\emptyset$, and let $\omega_N\in \Omega \coloneqq \{0,1\}^{\mathbb{Z}}$ be the indicator function of $A_N$. 
Let $\sigma $ denote the shift on the compact metric space $\Omega $, let $C=\{\omega \in \Omega \colon \omega (0)=1\}$, and set $\varrho_N=\frac 1N\sum_{n=1}^{N}\delta_{\sigma ^n\omega_N}$, a probability measure on $\Omega $. 
Then $\varrho_N(C)=\lvert A_N\rvert /N\geq \lambda $, while $\varrho_N(C\cap \sigma ^{-h}C)=0$ for every $h\in H_p$, as $A_N$ has no two elements differing by $h$. 
Let $\varrho $ be a weak$^\ast$ limit point of the $\varrho_N$; it is $\sigma $-invariant, and as $C$ and $C\cap \sigma ^{-h}C$ are clopen, 
$\varrho (C)\geq \lambda $ and $\varrho (C\cap \sigma ^{-h}C)=0$ for all $h\in H_p$. 
Since $\sigma $ is invertible, $\varrho (C\cap \sigma ^{-m}C)=0$ for every integer $m$ with $\lvert m\rvert \in H_p$, in particular for $m=p(k_n)$ with $n\notin Z$. 
As $Z$ is finite, the left side of \eqref{eENA}, computed in the system $(\Omega ,\varrho ,\sigma )$ with $\tilde A=C$, is therefore $0$. This contradicts \eqref{eENA}. 
That completes our proof. 
\end{proof}

\section{Uniform Distribution on \texorpdfstring{$[0,1]^2$}{[0,1]²}}% (fold)
\label{sec:uniformDistributionMod1}

% section ergodic_theorem (end)
\begin{definition} A sequence $(x_n)_{n \geq 1}$ in $[0,1) ^2 $ is said to be \emph{uniformly distributed} if for all $0\leq a < b < 1$ and $0\leq c<d<1$, 
\begin{equation}
    \lim_{N\to \infty} \Expectation _{[N]}  \mathbf{1}_{ (a,b) \times (c,d)} (x_n) = (b-a)(d-c) 
\end{equation}
That is, the proportion of integers $n$ with $x_n \in (a,b) \times (c,d)$ behaves as expected.  We say that a sequence $(x_n) \subset \mathbb{R} ^2 $  is \textit{uniformly distributed (mod $1$)} if the sequence of fractional parts $(\{ x_n \}) = ( x_{n,1} - \lfloor x_{n,1} \rfloor,  x_{n,2} - \lfloor x_{n,2} \rfloor )$ is uniformly distributed in $\mathbb{T}  ^2 $.      
\end{definition}

A polynomial $p\colon\mathbb{Z}^2\to\mathbb{R}$ is said to be \emph{irrational} if at least one of its non-constant monomials has an irrational coefficient.  

\begin{definition} A sequence $\{x_n \colon n \geq 0\}$ in $ \mathbb{Z} ^2  $ is said to be \emph{polynomial uniformly distributed, uniformly}, briefly \emph{PUDU}, if 
for every irrational polynomial $p\colon\mathbb{Z}^2\to\mathbb{R}$, 
\begin{equation}\label{e:PUDU}
    \lim_{N-M\to \infty} \Expectation _{M\leq n<N}    e( p (x_n))  = 0. 
\end{equation}
\end{definition}

By the Weyl criterion \eqref{eWeyl} below, if $(x_n)$ is PUDU, and $P=(p_1,p_2)\colon\mathbb{Z}^2\to\mathbb{R}^2$ is a polynomial map such that 
$m\cdot P$ is irrational for every $m\in\mathbb{Z}^2\setminus\{(0,0)\}$, then $P(x_n)$ is uniformly distributed (mod $1$), and indeed uniformly so: 
the averages in the definition of uniform distribution may be taken over $M\leq n<N$, with $N-M\to\infty$. 
The condition on $m\cdot P$ cannot be weakened to $P$ having an irrational coefficient: 
$P(x,y)=(\alpha x,2\alpha x)$ is not uniformly distributed (mod $1$), for any $\alpha$.

Let $ \mathcal{C}$ be a classical Cantor set with base $b$ and  digits set $ D$. 
Let $s_b(\cdot)$ be the sum of digits function with base $b$, that is, 
\begin{equation} \label{e:sumofdigits}
s_b\Bigl( \sum_{j=0}^k d_j b^j\Bigr) = 
\sum_{j=0}^k  d_j.  
\end{equation}
The main theorem of this section is this. 

\begin{theorem} \label{t:polyUniform}
Let $\mathcal C = \mathcal C _{b, D} =\{k_0< k_1 < \cdots \}$. 
Then the sequence $\bigl( (k_n, s_b(k_n) )\bigr)_{n\geq 0} $ is PUDU.  
\end{theorem}

Coquet \cites{MR506123,MR578819} has established polynomial uniform distribution for general Cantor sets. The novelty here is the uniformity, paired with the sum of digits function.  The uniformity is a property we need for a further work, in which we will necessarily need $0$ to be an allowed digit.

We will appeal to the Weyl criterion, which asserts that uniform distribution of $(x_n) \subset \mathbb{R} ^2 $ in $\mathbb{T} ^2 $ is equivalent to 
\begin{equation} \label{eWeyl}
    \lim_{N\to \infty} \Expectation _{[N]}  e(m \cdot x_n) =0 
    \qquad m\in \mathbb{Z} ^2 \setminus \{(0,0)\}. 
\end{equation}
Here, and throughout,  $e ( \alpha ) = e^{2 \pi i \alpha }$.  
Note that if $P$ is a polynomial map with $m\cdot P$ having only rational non-constant coefficients for some $m\neq(0,0)$, then  
letting $q$ be a common denominator of these coefficients, $e(qm\cdot P(k_n, s_b(k_n)))$ is constant, and so $P(k_n,s_b(k_n))$ fails the Weyl test for the frequency $qm$.  
Thus, the Theorem above is a characterization of the polynomial uniform distribution property.

The key lemma for this result, as well as the ergodic theorems in \S\ref{s;meanErgodic}, is below. 

%%%%%%%%%%%%%%%%%%%%%% Lemma 
\begin{lemma}  \label{l;irrational} 
Let $\mathcal{C} = \mathcal{C}_{b,D}=\{k_0<k_1<\cdots\}$, 
and let $s=\gcd(D)$ be the largest integer such that $D$ is contained in a progression of step size $s$. 
Call $(\alpha,\beta)\in\mathbb{R}^2$ \emph{exceptional} if 
\begin{equation}\label{e:exceptional}
s\alpha = \frac{a}{b-1}+\frac {r}{b^t} \quad\textup{and}\quad s\beta \equiv - \frac{a}{b-1} \mod 1, 
\qquad \textup{for some integers $a$, $r$, and $t\geq 0$.}
\end{equation}
(When $s=1$ this reads $\alpha = \frac{a}{b-1}+\frac {r}{b^t}$ and $\beta \equiv - \frac{a}{b-1}$.) 
Then the  limit 
\begin{equation}\label{ee:limit}
\lim _{N\to \infty} \Expectation 
_{0\leq n<N} e(k_n \alpha + s_b(k_n)\beta )
\end{equation}
exists for all $\alpha $ and $\beta $, and it is zero unless $(\alpha,\beta)$ is exceptional. 
Moreover, if $(\alpha,\beta)$ is not exceptional, and in particular if $\alpha $ or $\beta $ is irrational,   then   
\begin{equation} \label{e:uniformLimitExponentials} 
\lim _{N-M\to \infty} \Expectation 
_{ M\leq n<N} 
e(k_n \alpha + s_b(k_n)\beta )=0. 
\end{equation}
\end{lemma}

Limits of this type, for rational $\alpha$ and $\beta$, govern the distribution of $(k_n,s_b(k_n))$ in residue classes, which was studied in \cite{ERDOS199899}, under an additional restriction on the modulus. This restriction was recently removed in \cite{SAAVEDRA-ARAYA_2026}*{Theorem A}. 
Both papers establish rates of convergence to the uniform distribution, which we do not address.

\begin{remark} The integer $s$ enters \eqref{e:exceptional} because the function $\widehat L$ of \eqref{e;widehatL} below has modulus $1$ at each of the $s$ points $\theta \in [0,1)$ with $s\theta \in \mathbb{Z}$, and not at $\theta =0$ alone. 
It is for this reason that the proof begins by reducing to the case $s=1$. 
Consider for example the case of $b = 101$, and $D=\{0,5,10,\ldots,100\}$ the multiples of $5$ in $\{0,\ldots,100\}$, so that $s=5$. 
Then we have 
\begin{equation}
\widehat L (\alpha )  = \tfrac 1{21} \sum_{m=0} ^{20} e ( 5m \alpha ) , 
\end{equation}
which has modulus $1$ at $\alpha = 1/5, 2/5, 3/5, 4/5$.  
%A second example is  $b=100$, and $D = \{ 11, 31, 71 \}$. 
\end{remark}

\begin{proof} 
Throughout, let $\sigma_n=s_b(k_n)$ and $\Phi(n)=e(k_n\alpha+\sigma_n\beta)$. 
The key function is 
\begin{align} \label{e;widehatL}
\widehat L(\theta ) =  \Expectation _{d\in D} e(d\theta ) = \frac{1}{\lvert D\rvert} \sum_{d \in D} e(d \theta).
\end{align}
Since $0\in D$, we have $\lvert \widehat L(\theta)\rvert=1$ if and only if $e(d\theta)=1$ for all $d\in D$, that is, if and only if $s\theta\in\mathbb{Z}$. 
In particular, $\lvert\widehat L(\theta)\rvert<1$ except at the finitely many $\theta\in[0,1)$ with $s\theta\in\mathbb{Z}$. 

\smallskip 
\noindent\emph{Reduction to $s=1$.} 
If $s>1$, the digit set $D'= \{d/s \colon d\in D\}$ again satisfies the standing assumption, and $\gcd(D')=1$. 
The elements of $\mathcal C$ are $k_n=sk'_n$, where $k'_0<k'_1<\cdots$ are the elements of $\mathcal C'=\mathcal C_{b,D'}$, and $s_b(k_n)=s\,s_b(k'_n)$. 
Hence $e(k_n\alpha+s_b(k_n)\beta)=e(k'_n s\alpha+s_b(k'_n)s\beta)$, and all of our conclusions follow from the case $s=1$, applied to $(s\alpha,s\beta)$. 
(Note that $s\alpha$ is irrational if $\alpha$ is, and likewise for $\beta$.) 
So we assume from now on that $s=1$. 

\smallskip 
\noindent\emph{Products over blocks.} 
By Proposition \ref{p;CantorStructure}(i), as $j$ ranges over $0\leq j<\lvert D\rvert^\ell$, the base $b$ digits of $k_j$ in positions $0,\ldots,\ell-1$ range independently and uniformly over $D$. 
Hence 
\begin{align}  \label{e;widehat}
P_\ell \coloneqq 
\Expectation_{0\leq j<\lvert D\rvert^\ell}\Phi(j)
= \prod _{c=0}^{\ell-1} \widehat L (b^c\alpha + \beta  )  .
\end{align}
By Proposition \ref{p;CantorStructure}(ii), for all $i\geq 1$, $m\geq 0$, and $0\leq j<\lvert D\rvert^i$, 
\begin{equation}\label{e:blockSplit}
\Phi(\lvert D\rvert^im+j)=e\bigl(b^ik_m\alpha+\sigma_m\beta\bigr)\,\Phi(j). 
\end{equation}
Now let $0\leq M<N$ be integers, and $i\geq1$. Set $M_i=\lceil M/\lvert D\rvert^i\rceil$ and $N_i=\lfloor N/\lvert D\rvert^i\rfloor$. 
The integers in $[M,N)$ that do not lie in $[\lvert D\rvert^iM_i,\lvert D\rvert^iN_i)$ number at most $2\lvert D\rvert^i$. 
Summing \eqref{e:blockSplit} over the blocks $\{\lvert D\rvert^im+j\colon 0\leq j<\lvert D\rvert^i\}$, $M_i\leq m<N_i$, we obtain 
\begin{equation}\label{e:windowSplit}
\Expectation_{M\leq n<N}\Phi(n)
=\frac{\lvert D\rvert^i(N_i-M_i)}{N-M}\, P_i\, \Expectation_{M_i\leq m<N_i}e\bigl(b^ik_m\alpha+\sigma_m\beta\bigr) 
+O\Bigl(\frac{\lvert D\rvert^i}{N-M}\Bigr), 
\end{equation}
where the main term is interpreted as zero if $N_i\leq M_i$. 
The implied constant is absolute, and in particular does not depend on $\alpha$, $\beta$, $M$, $N$, or $i$. 

\smallskip 
\noindent\emph{The limit of $P_\ell$.} 
We show that $P_\ell \to 0$ as $\ell\to\infty$ unless $(\alpha,\beta)$ is exceptional, in which case $P_\ell=P_t$ for all $\ell\geq t$, where $t$ is as in \eqref{e:exceptional}. 
Let $\lVert\cdot\rVert$ denote the distance to the nearest integer. 

Suppose first that $\lVert b^c\alpha+\beta\rVert$ does not tend to zero as $c\to\infty$. 
Then there is an $\eta>0$ with $\lVert b^c\alpha+\beta\rVert\geq\eta$ for infinitely many $c$. 
As $s=1$, $\lvert\widehat L\rvert<1$ on the compact set $\{\theta\in\mathbb{T}\colon\lVert\theta\rVert\geq\eta\}$, so $\lvert\widehat L\rvert\leq 1-\kappa$ there, for some $\kappa>0$. 
Hence infinitely many factors in \eqref{e;widehat} have modulus at most $1-\kappa$, while all have modulus at most $1$, and $P_\ell\to0$. 

Suppose next that $\lVert b^c\alpha+\beta\rVert\to0$. 
Then $x=(b-1)\alpha$ satisfies $\lVert b^cx\rVert\leq \lVert b^{c+1}\alpha+\beta\rVert+\lVert b^c\alpha+\beta\rVert\to0$. 
We claim that $b^tx\in\mathbb{Z}$ for some $t\geq0$. 
Choose $c_0$ such that $\lVert b^cx\rVert<1/(2b)$ for all $c\geq c_0$. 
If $0<\lVert b^cx\rVert<1/(2b)$, then $\lVert b^{c+1}x\rVert=b\lVert b^cx\rVert$. 
So if $\lVert b^{c_0}x\rVert\neq0$, the quantities $\lVert b^{c_0+i}x\rVert=b^i\lVert b^{c_0}x\rVert$ would eventually exceed $1/(2b)$, a contradiction. 
Thus $b^{c_0}x\in\mathbb{Z}$, proving the claim with $t=c_0$. 
Since $\gcd(b-1,b^t)=1$, we can then write $\alpha=\frac{x}{b-1}=\frac{a}{b-1}+\frac{r}{b^t}$ with integers $a$ and $r$. 
As $b\equiv1\mod (b-1)$, for $c\geq t$ we have $b^c\alpha\equiv\frac{a}{b-1}\mod1$. 
So $\lVert b^c\alpha+\beta\rVert\to0$ forces $\beta\equiv-\frac a{b-1}\mod 1$, that is, $(\alpha,\beta)$ is exceptional. 
Moreover $b^c\alpha+\beta\in\mathbb{Z}$ for all $c\geq t$, so the factors in \eqref{e;widehat} with $c\geq t$ equal $\widehat L(0)=1$, and $P_\ell=P_t$ for $\ell\geq t$. 

Conversely, if $(\alpha,\beta)$ is exceptional, the same computation shows $b^c\alpha+\beta\in\mathbb{Z}$ for $c\geq t$. 
This completes the analysis of $P_\ell$. 

\smallskip 
\noindent\emph{Non-exceptional $(\alpha,\beta)$.} 
Let $\epsilon>0$, and choose $i$ with $\lvert P_i\rvert<\epsilon$. 
Since $\lvert D\rvert^i(N_i-M_i)\leq N-M$, \eqref{e:windowSplit} gives 
\begin{equation}
\Bigl\lvert\Expectation_{M\leq n<N}\Phi(n)\Bigr\rvert\leq \epsilon+O\Bigl(\frac{\lvert D\rvert^i}{N-M}\Bigr). 
\end{equation}
Letting $N-M\to\infty$, and then $\epsilon\to0$, proves \eqref{e:uniformLimitExponentials}. 
Taking $M=0$, the limit \eqref{ee:limit} exists and equals zero. 
If $\alpha$ or $\beta$ is irrational, then $(\alpha,\beta)$ is not exceptional, since exceptional pairs are rational. 

\smallskip 
\noindent\emph{Exceptional $(\alpha,\beta)$.} 
Let $t$ be as in \eqref{e:exceptional}, and take $i\geq t$. 
Then $b^i\alpha\equiv\frac{a}{b-1}\equiv-\beta\mod1$, so for every $m$, 
\begin{equation}
b^ik_m\alpha+\sigma_m\beta\equiv \frac{a}{b-1}\,(k_m-\sigma_m)\equiv 0\mod 1, 
\end{equation}
because $b-1$ divides $b^c-1$ for every $c$, and hence divides $k_m-s_b(k_m)$. 
Thus the average over $m$ in \eqref{e:windowSplit} equals $1$, and $P_i=P_t$. 
Taking $M=0$ in \eqref{e:windowSplit}, and using $\lvert D\rvert^iN_i/N\to1$ as $N\to\infty$, we obtain 
\begin{equation}
\lim_{N\to\infty}\Expectation_{0\leq n<N}\Phi(n)=P_t. 
\end{equation}
So the limit \eqref{ee:limit} exists in this case as well. 
\end{proof} 

\begin{proof}[Proof of Theorem \ref{t:polyUniform}] 
Throughout the proof we write $\sigma_n=s_b(k_n)$. 
By definition of PUDU, we must show that for every irrational polynomial $p\colon \mathbb{Z}^2\to\mathbb{R}$, 
\begin{equation}  \label{e;UDtoshow}
\lim_{N-M\to \infty} \Expectation_{M\leq n<N} 
e(p(k_n, \sigma_n)) =0. 
\end{equation}
Following Coquet \cite{MR506123}, 
say that $p \colon \mathbb{Z} ^2 \to \mathbb{R} $ has \emph{irrational degree $\ell$} 
if it has an irrational coefficient on a term of degree $\ell$, and all higher degree terms, if they are present, have  rational coefficients.  
An irrational polynomial has irrational degree $\ell\geq 1$. 

Our induction hypothesis is $\textup{UD} _{\ell}$:  The limit \eqref{e;UDtoshow} holds for all polynomials of irrational degree $\ell$.
  Note that to establish $\textup{UD} _{\ell}$, it suffices to consider polynomials with no constant term, which we assume below.   

\begin{proof}[Proof of $\textup{UD} _{1}$]
Consider first linear polynomials. If $\alpha$ or $\beta$ is irrational, then 
\eqref{e:uniformLimitExponentials} gives
\begin{equation} \label{e:deltaProgressions} 
\lim _{N-M\to \infty} \Expectation 
_{ M\leq n<N } 
e(k_n \alpha + \sigma_n\beta )=0.  
\end{equation}
That is, we have uniform convergence  for linear polynomials with an irrational coefficient.  

Now consider a general polynomial $p$ of irrational degree $1$. Thus $p(x,y) =\alpha x+\beta y+p'(x,y)$, where 
either $\alpha$ or $\beta$ is irrational, and  $p'$ has exclusively rational coefficients and terms of degree two or higher. 
Let $q$ be the least  common multiple  of the  denominators of the coefficients of $p'$.  
For integers $x,y,t,t'$, every monomial satisfies $(x+qt)^i(y+qt')^{i'}\equiv x^iy^{i'} \mod q$, so 
$e(p'(x,y))$ depends only on $(x \mod q,\ y\mod q)$. 
Define 
\begin{equation} \label{e:Nrq}
  N(r, r') \coloneqq  \{ n \colon k_n \equiv r \mod q ,\  \sigma_n \equiv r'\mod q \} , \qquad 0\leq r,r'<q. 
\end{equation}
The indicator of this set has the Fourier expansion 
\begin{equation}
    \mathbf{1}_{N(r,r')} (n) = \Expectation_{0\leq u,u'<q} 
    e\bigl(u(k_n-r)/q+u'(\sigma_n-r')/q\bigr) . 
\end{equation}
And so, we have 
\begin{align}
    \Expectation 
_{ M \leq n < N } 
e( p(k_n, \sigma_n))
&= \sum_{0\leq r,r'<q} e(p'(r,r'))
    \Expectation 
_{ M \leq n < N } \mathbf{1}_{N(r,r')} (n) 
e( \alpha k_n+\beta \sigma_n)
\\
&= 
\sum_{0\leq r,r'<q} e(p'(r,r'))
 \Expectation_{0\leq u,u'<q} e\bigl(-(ur+u'r')/q\bigr)
 \\ & \qquad \times  \Expectation_{M\leq n < N}  e\bigl( (\alpha+\tfrac uq)k_n+(\beta+\tfrac{u'}q)\sigma_n\bigr). 
\end{align}
For each of the finitely many choices of $u,u'$, one of $\alpha+\frac uq$ and $\beta+\frac{u'}q$ is irrational, 
so the innermost average tends to zero by \eqref{e:deltaProgressions}. 
\end{proof}

For the induction step, we need to know when differencing lowers the irrational degree by exactly one. 
For a polynomial $Q$ with real coefficients, let $\overline Q$ denote the polynomial obtained by reducing each coefficient modulo $\mathbb{Q}$. 
That is, $\overline Q$ has coefficients in the $\mathbb{Q}$-vector space $\mathbb{R}/\mathbb{Q}$, and $\overline Q=0$ if and only if all coefficients of $Q$ are rational. 
Derivatives of $\overline Q$ are defined as usual; for $w=(w_1,w_2)\in\mathbb{Q}^2$ set $\partial_w=w_1\partial_x+w_2\partial_y$. 

\begin{lemma}\label{l:badDirections}
Let $Q$ be a homogeneous polynomial of degree $\ell\geq 1$ with $\overline Q\neq 0$. 
There are integers $(c_1,c_2)\neq(0,0)$ such that for all $w\in\mathbb{Z}^2$, if $\partial_w\overline Q=0$, then $c_1w_1+c_2w_2=0$. 
\end{lemma}

\begin{proof}
The map $w\mapsto \partial_w\overline Q$ is $\mathbb{Q}$-linear on $\mathbb{Q}^2$. 
If its kernel were all of $\mathbb{Q}^2$, then $\partial_x\overline Q=\partial_y\overline Q=0$, so $\overline Q$ would be constant. 
As $\overline Q$ is homogeneous of degree $\ell\geq1$, this forces $\overline Q=0$, a contradiction. 
Hence the kernel has dimension at most one, and is contained in a line $\{c_1w_1+c_2w_2=0\}$, with $c_1,c_2$ integers, not both zero. 
\end{proof}

\begin{proof}[Proof of $\textup{UD}_{\ell-1} \implies \textup{UD}_\ell$] 
For $\ell\geq 2$, we assume $\textup{UD}_{\ell-1}$. 
Fix a polynomial $p$ of irrational degree $\ell$, and let $Q$ be its homogeneous part of degree $\ell$, so that $\overline Q\neq 0$. 
Let $(c_1,c_2)$ be as in Lemma \ref{l:badDirections}. 

Fix an integer $\rho\geq 1$. 
Given $M<N$, let $M'=\lceil M/\lvert D\rvert^{\rho}\rceil$ and $N'=\lfloor N/\lvert D\rvert^{\rho}\rfloor$. 
The integers $n\in[M,N)$ that are not in $[\lvert D\rvert^{\rho}M',\lvert D\rvert^{\rho}N')$ number at most $2\lvert D\rvert^{\rho}$, and $N'-M'\to\infty$ as $N-M\to\infty$. 
Each $n\in[\lvert D\rvert^{\rho}M',\lvert D\rvert^{\rho}N')$ is uniquely written as $n=\lvert D\rvert^{\rho}m+j$, with $M'\leq m<N'$ and $0\leq j<\lvert D\rvert^{\rho}$, 
and by Proposition \ref{p;CantorStructure}, 
\begin{equation}
k_n=b^{\rho}k_m+k_j, \qquad \sigma_n=\sigma_m+\sigma_j. 
\end{equation}
Therefore, by the Cauchy--Schwarz inequality in $m$, 
\begin{align}
\Bigl\lvert \Expectation_{M\leq n<N}e(p(k_n,\sigma_n))\Bigr\rvert^2
&\leq 
\Expectation_{M'\leq m<N'}\Bigl\lvert \Expectation_{0\leq j<\lvert D\rvert^{\rho}} e\bigl(p(b^{\rho}k_m+k_j,\sigma_m+\sigma_j)\bigr)\Bigr\rvert^2 + O\Bigl(\frac{\lvert D\rvert^{\rho}}{N-M}\Bigr)
\\ \label{e:CSblocks}
&\leq 
\Expectation_{0\leq j,j'<\lvert D\rvert^{\rho}}\Bigl\lvert 
\Expectation_{M'\leq m<N'}e\bigl(R_{j,j'}(k_m,\sigma_m)\bigr)\Bigr\rvert + O\Bigl(\frac{\lvert D\rvert^{\rho}}{N-M}\Bigr), 
\end{align}
where 
\begin{equation}
R_{j,j'}(X,Y)\coloneqq p(b^{\rho}X+k_j,\,Y+\sigma_j)-p(b^{\rho}X+k_{j'},\,Y+\sigma_{j'}). 
\end{equation}

We claim that $R_{j,j'}$ has irrational degree $\ell-1$, unless $w=w_{j,j'}\coloneqq(k_j-k_{j'},\,\sigma_j-\sigma_{j'})$ satisfies $c_1w_1+c_2w_2=0$. 
Indeed, the terms of $p$ of degree larger than $\ell$ have rational coefficients, and so contribute only rational coefficients to $R_{j,j'}$. 
The terms of degree $\ell$ in $R_{j,j'}$ coming from $Q$ cancel, and the terms of $p$ of degree less than $\ell$ contribute only terms of degree at most $\ell-2$. 
Hence, modulo rational coefficients, the part of $R_{j,j'}$ of degree $\ell-1$ is 
\begin{equation}
(\partial_w\overline Q)(b^{\rho}X,Y). 
\end{equation}
The coefficient of $X^iY^{\ell-1-i}$ here is $b^{\rho i}$ times the corresponding coefficient of $\partial_w\overline Q$, so it is non-zero if and only if $\partial_w\overline Q\neq 0$. 
By Lemma \ref{l:badDirections}, this holds unless $c_1w_1+c_2w_2=0$, proving the claim. 

Call the pair $(j,j')$ \emph{bad} if $c_1k_j+c_2\sigma_j=c_1k_{j'}+c_2\sigma_{j'}$, and \emph{good} otherwise. 
For each of the finitely many good pairs $(j,j')$, the polynomial $R_{j,j'}$ has irrational degree $\ell-1\geq 1$. 
By $\textup{UD}_{\ell-1}$, the corresponding inner average in \eqref{e:CSblocks} tends to zero as $N'-M'\to\infty$. 
Bounding the inner average by $1$ for bad pairs, we conclude that 
\begin{equation}
\limsup_{N-M\to\infty}\Bigl\lvert \Expectation_{M\leq n<N}e(p(k_n,\sigma_n))\Bigr\rvert^2\leq \eta_\rho, 
\end{equation}
where $\eta_\rho$ is the proportion of bad pairs among all pairs $0\leq j,j'<\lvert D\rvert^{\rho}$. 
Since $\rho$ is arbitrary, it remains to show that $\eta_\rho\to0$ as $\rho\to\infty$. 
Note that for $0\leq j<\lvert D\rvert^{\rho}$, the element $k_j$ has at most $\rho$ digits, so $0\leq\sigma_j\leq (b-1)\rho$. 
\begin{itemize}
\item If $c_2=0$, then $c_1\neq0$, and a bad pair has $k_j=k_{j'}$, that is $j=j'$. So $\eta_\rho=\lvert D\rvert^{-\rho}$. 

\item If $c_1\neq0$ and $c_2\neq 0$, a bad pair has $\lvert k_j-k_{j'}\rvert = \lvert c_2/c_1\rvert\,\lvert\sigma_j-\sigma_{j'}\rvert\leq \lvert c_2/c_1\rvert(b-1)\rho$. 
As the $k_j$ are distinct integers, for each $j$ there are at most $2\lvert c_2/c_1\rvert(b-1)\rho+1$ such $j'$. 
So $\eta_\rho\leq (2\lvert c_2/c_1\rvert(b-1)\rho+1)\lvert D\rvert^{-\rho}$. 

\item If $c_1=0$, a bad pair has $\sigma_j=\sigma_{j'}$. 
For $j$ uniformly distributed on $\{0,\ldots,\lvert D\rvert^{\rho}-1\}$, the base $\lvert D\rvert$ digits of $j$ are independent and uniform, 
so $\sigma_j$ is the sum of $\rho$ independent random variables, each uniformly distributed on $D$. 
Thus $\Expectation_{0\leq j<\lvert D\rvert^{\rho}}e(\theta\sigma_j)=\widehat L(\theta)^{\rho}$, with $\widehat L$ as in \eqref{e;widehatL}, and 
\begin{equation}
\eta_\rho=\int_0^1\Bigl\lvert \Expectation_{0\leq j<\lvert D\rvert^{\rho}}e(\theta\sigma_j)\Bigr\rvert^2\;d\theta
=\int_0^1\lvert\widehat L(\theta)\rvert^{2\rho}\;d\theta. 
\end{equation}
As shown in the proof of Lemma \ref{l;irrational}, $\lvert\widehat L(\theta)\rvert<1$ except for the finitely many $\theta\in[0,1)$ with $s\theta\in\mathbb{Z}$. 
By dominated convergence, $\eta_\rho\to0$. 
\end{itemize}
In all cases $\eta_\rho\to0$, which completes the proof of $\textup{UD}_{\ell-1}\implies\textup{UD}_\ell$. 
\end{proof}

By induction, $\textup{UD}_\ell$ holds for all $\ell\geq1$, which proves the Theorem. 
\end{proof}

\section{Uniform Distribution on \texorpdfstring{$\mathbb Z_a \times \mathbb Z_{a'}$}{Za x Za'}} \label{sec:modular}
While uniform distribution results have been investigated in \cite{SAAVEDRA-ARAYA_2026}, we choose to revisit those with a simpler approach through the lens of Weyl's equidistribution theorem combined with Lemma \ref{l;irrational}. 

First, for the sum of digits function.  

\begin{proposition}  \label{p;sumDigits}
Let $\mathcal{C} = \mathcal{C}_{b,D}=\{k_0<k_1<\cdots \}$, let $s$ be the largest integer such that $D$ is contained in a progression of step size $s$, 
and let $q\geq 1$ be an integer. Set $g=\textup{gcd}(q,s)$. 
Then the sequence 
$ \{s_b(k_n) \mod q\}$ is asymptotically uniformly distributed on the subgroup $g\mathbb{Z}_q$ of $\mathbb{Z} _{q}$. 
In particular, it is asymptotically uniformly distributed on $\mathbb{Z}_q$ if and only if $\textup{gcd}(q,s)=1$. 
\end{proposition}

\begin{proof}
Every digit is a multiple of $s$, so $s_b(k_n)\in s\mathbb{Z}\subset g\mathbb{Z}$ for all $n$. 
The uniform distribution on the subgroup $g\mathbb{Z}_q$, which has order $q/g$, has Fourier coefficients 
$\Expectation_{x\in g\mathbb{Z}_q}e(ux/q)=\mathbf{1}_{(q/g)\mid u}$, for $u\in\mathbb{Z}_q$. 
So it suffices to show that for every integer $u$, 
\begin{align}\label{e:sumDigitsFourier}
\lim_{N\to\infty} 
\Expectation_{0\leq n<N}e\Bigl(s_b(k_n) \frac{u}{q}\Bigr) =\mathbf{1}_{(q/g)\mid u}. 
\end{align}
If $(q/g)\mid u$, then $s_b(k_n)u/q=\frac{s_b(k_n)}{g}\cdot\frac{u}{q/g}\in\mathbb{Z}$, and \eqref{e:sumDigitsFourier} is clear. 
Otherwise, we apply Lemma \ref{l;irrational} with $(\alpha,\beta)=(0,u/q)$; the limit exists, and it is zero unless $(0,u/q)$ is exceptional. 
If it were exceptional, then $0=\frac{a}{b-1}+\frac{r}{b^t}$ for some integers $a,r,t$. As $\gcd(b-1,b^t)=1$, this forces $\frac{a}{b-1}\in\mathbb{Z}$, 
and then $\frac{su}{q}\equiv-\frac{a}{b-1}\equiv 0\mod 1$. That is, $q\mid su$, so $(q/g)\mid (s/g)u$, and since $\gcd(q/g,s/g)=1$, $(q/g)\mid u$, a contradiction. 
\end{proof}

Let us discuss the case of just the elements of the Cantor set.  Uniform distribution in this case is much harder to characterize.  Recall that $0$ is an allowed digit, by our standing assumption; this is important for some of the combinatorial and ergodic theoretic results.

\begin{lemma} \label{l:Allq}
Let $\mathcal C = \mathcal C_{b,D} 
= \{k_0<k_1 < \cdots\}$, and let $s$ be the largest integer such that $D$ is contained in a progression of step size $s$. 
Then, for all $q\in \mathbb{N}   $, 
the sequence $ k_n\mod q$ converges to a distribution on $\mathbb{Z} _q$. 
That is, for all $x\in \mathbb{Z} _q$ the limit 
\begin{equation}
\pi_q (x) = \lim_{N\to \infty} \Expectation _{0\leq n<N} 
\mathbf{1} _{k_n \equiv x \mod q} 
\end{equation}
exists. Moreover, we have 
\begin{enumerate}
    \item If $\gcd(q,b)=1$, then $\pi_q$ is the uniform distribution on the subgroup $g\mathbb{Z}_q$, where $g=\gcd(q,s)$. 
    In particular, $k_n \mod q $ is asymptotically uniformly distributed on $\mathbb{Z}_q$ if and only if $\gcd(q,s)=1$. 

    \item  If  $q\mid b^j$, for some integer $j\geq 1$,   
    then $\pi_q$ is the distribution of $k_i\mod q$, for $i$ uniformly distributed on $\{0,1,\ldots,\lvert D\rvert^j-1\}$. 
    That is, $\pi_q$ is the distribution of $\mathcal C\cap[0,b^j)\mod q$. 
    
    \item  In general, write $q=ru$, where $\gcd(r,b)=1$  and every prime factor of $u$ divides $b$. 
    Then, under the isomorphism $\mathbb{Z}_q\cong\mathbb{Z}_r\times\mathbb{Z}_u$, $\pi_q$ is the product measure $\pi_r\times\pi_u$, 
    where $\pi_r$ is described in (1), and $\pi_u$ in (2). 
\end{enumerate}
\end{lemma}

Note in particular, that if  $q\mid b^j$, for some integer $j\geq 1$,   the Cantor set could be uniformly distributed on $\mathbb{Z} _{q}$, depending on the choice of $D$, and the smallest choice of $j$ with $q \mid b^j$.

\begin{proof}
By Lemma \ref{l;irrational}, the limit $\lim_{N}\Expectation_{0\leq n<N}e(k_n\frac aq)$ exists for every integer $a$. 
By Fourier inversion on $\mathbb{Z}_q$, the limit $\pi_q(x)$ exists for every $x$. 

For (1), let $\gcd(q,b)=1$, and $g=\gcd(q,s)$. As in the proof of Proposition \ref{p;sumDigits}, it suffices to show that 
$\lim_N\Expectation_{0\leq n<N}e(k_n\frac aq)=\mathbf{1}_{(q/g)\mid a}$. 
If $(q/g)\mid a$, then $k_na/q\in\mathbb{Z}$, since $g\mid s\mid k_n$. 
Otherwise, by Lemma \ref{l;irrational} the limit is zero unless $(a/q,0)$ is exceptional. 
If it were, then $s\frac aq=\frac{c}{b-1}+\frac{r'}{b^t}$ and $0\equiv-\frac{c}{b-1}\mod 1$, for some integers $c,r',t$. 
Then $\frac{c}{b-1}\in\mathbb{Z}$, so $b^t\frac{sa}{q}\in\mathbb{Z}$. As $\gcd(q,b)=1$, we get $q\mid sa$, and so $(q/g)\mid a$, a contradiction. 

For (2), suppose that $q\mid b^j$. 
By Proposition \ref{p;CantorStructure}, for $m\geq0$ and $0\leq i<\lvert D\rvert^j$, $k_{\lvert D\rvert^jm+i}=b^jk_m+k_i\equiv k_i\mod q$. 
So on each block $\{\lvert D\rvert^jm+i\colon 0\leq i<\lvert D\rvert^j\}$, the residues $k_n\mod q$ have exactly the distribution of $k_i\mod q$, $0\leq i<\lvert D\rvert^j$. 
The incomplete block contributes $O(\lvert D\rvert^j/N)$, and (2) follows. 

For (3), fix $j\geq1$ with $u\mid b^j$, and let $x\in\mathbb{Z}_r$, $y\in\mathbb{Z}_u$. 
Writing $n=\lvert D\rvert^jm+i$ as above, we have $k_n\equiv k_i\mod u$, and $k_n=b^jk_m+k_i$. 
As $b$ is invertible modulo $r$, the condition $k_n\equiv x\mod r$ is equivalent to $k_m\equiv b^{-j}(x-k_i)\mod r$. 
Hence, up to an error $O(\lvert D\rvert^j/N)$, 
\begin{equation}
\Expectation_{0\leq n<N}\mathbf{1}_{k_n\equiv x\mod r}\mathbf{1}_{k_n\equiv y\mod u}
=\Expectation_{0\leq i<\lvert D\rvert^j}\mathbf{1}_{k_i\equiv y\mod u}\;
\Expectation_{0\leq m<\lfloor N/\lvert D\rvert^j\rfloor}\mathbf{1}_{k_m\equiv b^{-j}(x-k_i)\mod r}. 
\end{equation}
As $N\to\infty$, the inner average tends to $\pi_r(b^{-j}(x-k_i))$. 
By (1), $\pi_r$ is uniform on $g'\mathbb{Z}_r$, where $g'=\gcd(r,s)$. 
Since $k_i\in s\mathbb{Z}\subset g'\mathbb{Z}$, and multiplication by the unit $b^{-j}$ preserves $g'\mathbb{Z}_r$, we have $\pi_r(b^{-j}(x-k_i))=\pi_r(x)$. 
Therefore the limit is $\pi_r(x)\,\pi_u(y)$, using (2) for $\pi_u$. 
\end{proof}

In the particular case where $a=0$ (that is, elements of the Cantor set divisible by $q$), one has strictly positive density along this residue class.

\begin{lemma}\label{lem:positiveDensity}
    Let $\mathcal{C}=\mathcal C_{b,D}=\{k_0<k_1<\cdots\}$. Then, for any $q\in \mathbb{N}$, \begin{align*}
        \lim_N \Expectation_{0\leq n<N}\mathbf{1}_{k_n\equiv 0\ (\text{mod }q)}>0.
    \end{align*}
\end{lemma}

\begin{proof}
Write $q=ru$ as in Lemma \ref{l:Allq}(3). Then $\pi_q(0)=\pi_r(0)\,\pi_u(0)$. 
By Lemma \ref{l:Allq}(1), $\pi_r(0)=\gcd(r,s)/r>0$. 
By Lemma \ref{l:Allq}(2), with $j$ such that $u\mid b^j$, $\pi_u(0)\geq \lvert D\rvert^{-j}>0$, because $k_0=0$. 
\end{proof}

We record here a result about the  distribution of 
the pair $(k, s_b(k))$ in the modular case, for integers $k$ in a Cantor set, where $s_b(k)$ is the sum of digits function. 

\begin{theorem}\label{t:modularJoint}
Let $\mathcal{C} = \mathcal{C}_{b,D}=\{k_0<k_1<\cdots \}$. Let $s$ be the largest integer such that $D$ is contained in a progression of step size $s$.  Let $a,a'\geq 1$ be integers, and consider the distributions 
\begin{equation}
    (k_n \mod a,\ s_b(k_n)\mod a')\in\mathbb{Z}_a\times\mathbb{Z}_{a'}, \quad 0\leq n< N. 
\end{equation}
These conclusions hold. 
\begin{enumerate}
    \item  The limiting distribution, as $N\to\infty$, exists. 

    \item The limiting distribution is uniform on $\mathbb{Z}_a\times\mathbb{Z}_{a'}$ if 
    $\textup{gcd}(a,sb)=1$ and  $\textup{gcd}(a',s(b-1))=1$. 

    \item The limiting distribution is not uniform if $\textup{gcd}(a,s)>1$, or $\textup{gcd}(a',s)>1$, or $\textup{gcd}(a,a',b-1)>1$. 
\end{enumerate}
\end{theorem}

The sufficient condition in (2) is not necessary. 
For instance, with $b=4$ and $D=\{0,1\}$, $k_n\mod 2$ is the last digit of $k_n$, which is asymptotically uniform on $\mathbb{Z}_2$, although $\gcd(2,b)>1$. 
And $s_b(k_n)\mod 3$ is asymptotically uniform on $\mathbb{Z}_3$ by Proposition \ref{p;sumDigits}, although $\gcd(3,b-1)>1$. 

\begin{proof}
For (1), by Lemma \ref{l;irrational} the limit 
\begin{equation}
\lim_{N\to\infty}\Expectation_{0\leq n<N}e\Bigl(k_n\frac ua+s_b(k_n)\frac{u'}{a'}\Bigr)
\end{equation}
exists for all integers $u,u'$. By Fourier inversion on $\mathbb{Z}_a\times\mathbb{Z}_{a'}$, the limiting distribution exists. 

For (2), by the Weyl criterion on $\mathbb{Z}_a\times\mathbb{Z}_{a'}$, it suffices to show that the limits above are zero for $(u,u')\not\equiv(0,0)\mod (a,a')$. 
By Lemma \ref{l;irrational}, it suffices to show that $(u/a,u'/a')$ is not exceptional. 
Suppose it is, so that for some integers $c$, $r$ and $t$, 
\begin{equation}
s\frac ua=\frac{c}{b-1}+\frac{r}{b^t}, \qquad s\frac{u'}{a'}+\frac{c}{b-1}\in\mathbb{Z}. 
\end{equation}
The number $x=su'/a'$ satisfies $a'x\in\mathbb{Z}$, and $y=c/(b-1)$ satisfies $(b-1)y\in\mathbb{Z}$. 
As $x+y\in\mathbb{Z}$, both $a'y$ and $(b-1)y$ are integers, and since $\gcd(a',b-1)=1$, $y\in\mathbb{Z}$; hence also $x\in\mathbb{Z}$. 
So $a'\mid su'$, and as $\gcd(a',s)=1$, $a'\mid u'$. 
Then $b^tsu/a\in\mathbb{Z}$, and as $\gcd(a,sb)=1$, $a\mid u$. This is a contradiction. 

For (3), if a prime $p$ divides $\gcd(a,s)$, then $k_n\equiv 0\mod p$ for all $n$, as every digit is a multiple of $s$; so the distribution is not uniform. 
The same argument applies to $s_b(k_n)$ if $p$ divides $\gcd(a',s)$. 
If a prime $p$ divides $\gcd(a,a',b-1)$, then $k_n\equiv s_b(k_n)\mod p$ for all $n$, since $b-1$ divides $k-s_b(k)$ for every $k$. 
So the image of the distribution in $\mathbb{Z}_p\times\mathbb{Z}_p$ is supported on the diagonal, and it is not uniform. 
\end{proof}

\section{Ergodic Theorems} \label{s;meanErgodic}
We specialize the discussion of the previous section to that of mean ergodic theorems along Cantor sets of integers.   The core is this Theorem for norm convergence. 

\begin{theorem} \label{t:ergodic}
Let $\mathcal C = \mathcal C _{b,D}=\{k_0<k_1<\cdots\}$, and let $s=\gcd(D)$ be the largest  step size of a progression that contains $D$. 
Then,  for commuting unitary operators $(U,V)$ on a Hilbert space $\mathcal H$, and all $x\in \mathcal H$, 
\begin{equation} \label{e:ergodicLimit}
  \lim_N  \Expectation_{0\leq n<N} U^{k_n} V^{s_b(k_n)} x\quad \textup{exists in norm}. 
 \end{equation}
Let $\mathcal E\subset \mathbb T ^2 $ be the set of exceptional pairs \eqref{e:exceptional}, namely those $(\alpha ,\beta )$ for which 
$s\alpha \equiv \frac{a}{b-1}+\frac {r}{b^t}$ and $s\beta \equiv - \frac{a}{b-1}$ modulo $1$, for some integers $a$, $r$ and $t\geq 0$. 
It is countable. 
For $(\alpha ,\beta )\in \mathcal E$, let $P_{\alpha ,\beta }$ denote the orthogonal projection onto the joint eigenspace 
$\{y\in \mathcal H\colon Uy=e(\alpha )y, \ Vy=e(\beta )y\}$. 
Then the limit in \eqref{e:ergodicLimit} equals 
\begin{equation} \label{e:ergodicFormula}
    \sum_{(\alpha ,\beta )\in \mathcal E} F(\alpha ,\beta ) \, P_{\alpha ,\beta } x , 
\end{equation}
where $F(\alpha ,\beta )$ denotes the limit \eqref{ee:limit}. 
In particular, the limit depends only on the joint point spectrum of $(U,V)$ on $\mathcal E$, and it is zero if $(U,V)$ has no joint eigenvalue in $\mathcal E$. 
\end{theorem}

\begin{proof}
Let $E$ be the joint spectral measure of the commuting unitaries $(U,V)$ on $\mathbb T ^2 $, so that $U^mV^{m'}=\int_{\mathbb T ^2 }e(m\alpha +m'\beta )\;dE(\alpha ,\beta )$ 
for integers $m,m'$, and put $\nu_x(\cdot )=\langle E(\cdot )x,x\rangle$, a finite positive measure. 
Set $F_N(\alpha ,\beta )=\Expectation_{0\leq n<N}e(k_n\alpha +s_b(k_n)\beta )$, so that 
\begin{equation}
\Bigl\lVert \Expectation_{0\leq n<N} U^{k_n}V^{s_b(k_n)}x - \Expectation_{0\leq n<M} U^{k_n}V^{s_b(k_n)}x \Bigr\rVert ^2 
= \int_{\mathbb T ^2 } \lvert F_N - F_M\rvert ^2 \; d \nu_x . 
\end{equation}
By Lemma \ref{l;irrational}, the functions $F_N$ converge pointwise on $\mathbb T ^2 $ to a function $F$ which vanishes off $\mathcal E$. 
As $\lvert F_N\rvert \leq 1$, dominated convergence shows that the averages are Cauchy in $\mathcal H$, proving \eqref{e:ergodicLimit}, 
and the same estimate identifies the limit as $\bigl(\int_{\mathbb T ^2 }F\;dE\bigr)x$. 

The set $\mathcal E$ is countable: there are countably many values of $\frac{a}{b-1}+\frac{r}{b^t}$, each giving at most $s$ values of $\alpha$, 
and then $\beta$ is determined modulo $1$ by $a$ modulo $b-1$, again up to $s$ choices. 
Since $F$ vanishes off the countable set $\mathcal E$, 
\begin{equation}
\int_{\mathbb T ^2 } F \; dE = \sum_{(\alpha ,\beta )\in \mathcal E} F(\alpha ,\beta )\, E(\{(\alpha ,\beta )\}) , 
\end{equation}
the series converging in the strong operator topology. 
Finally, $E(\{(\alpha ,\beta )\})$ is the projection onto $\{y\colon Uy=e(\alpha )y,\ Vy=e(\beta )y\}$, which is $P_{\alpha ,\beta }$. 
This proves \eqref{e:ergodicFormula}. 
\end{proof}

We are interested in polynomial versions of the result above, which we can obtain under stronger  assumptions.

\begin{theorem} \label{t:polyMeanErgodic}
Let $\mathcal C = \mathcal C _{b,D}=\{k_0<k_1<\cdots\}$.  Then for all commuting unitary operators, $(U,V)$ on a Hilbert space $\mathcal H$, we have for any integer-valued polynomial $p \colon \mathbb{Z} ^2  \to \mathbb{Z} ^2 $, 
\begin{equation}
  \lim_N  \Expectation_{0\leq n<N}(U,V) ^{p (k_n, s_b(k_n))}  x  
  \quad \textup{exists in norm} 
 \end{equation}
for all $x\in \mathcal H$. 
Above, by $(U,V) ^{(m,m')}$ we mean $U ^{m}V ^{m'}$.  
\end{theorem}

\begin{proof} 
The central point is that for all $\alpha \in  \mathbb T ^2  $, the limit below exists: 
\begin{align}
    F(\alpha)=\lim_N  F_N(\alpha), \qquad F_N(\alpha)=\Expectation_{0\leq n<N} e(\alpha \cdot  p(k_n, s_b(k_n) ) ) . 
\end{align}
Write $\alpha\cdot p=c_0+\tilde p$, where $c_0$ is the constant term of the real polynomial $\alpha\cdot p$. 
If $\tilde p$ has an irrational coefficient, then $F_N(\alpha)\to0$ by Theorem \ref{t:polyUniform}. 
Otherwise, all coefficients of $\tilde p$ are rational; let $q$ be a common denominator. 
As in the proof of $\textup{UD}_1$ in Theorem \ref{t:polyUniform}, $e(\tilde p(x,y))$ depends only on $(x\mod q,\ y\mod q)$. 
Hence 
\begin{equation}
F_N(\alpha)=e(c_0)\sum_{x,y\in\mathbb{Z}_q}e(\tilde p(x,y))\,\Expectation_{0\leq n<N}\mathbf{1}_{k_n\equiv x\mod q}\,\mathbf{1}_{s_b(k_n)\equiv y\mod q}, 
\end{equation}
which converges as $N\to\infty$ by Theorem \ref{t:modularJoint}(1). 

By the spectral theorem for the commuting unitary operators $U$ and $V$, there is a finite positive measure $\nu_x$ on $\mathbb{T}^2$ such that 
for all $N,M$, 
\begin{equation}
\Bigl\lVert \Expectation_{0\leq n<N}(U,V) ^{p (k_n, s_b(k_n))}  x-\Expectation_{0\leq n<M}(U,V) ^{p (k_n, s_b(k_n))}  x\Bigr\rVert^2
=\int_{\mathbb{T}^2}\lvert F_N(\alpha)-F_M(\alpha)\rvert^2\;d\nu_x(\alpha). 
\end{equation}
The functions $F_N$ are bounded by $1$ and converge pointwise, so by dominated convergence the averages form a Cauchy sequence in $\mathcal H$. 
\end{proof}

The proof above is purely an $L^2$ result. 
It applies to an \emph{invertible} measure preserving system $(X,\mathcal A,\mu ,T)$: take $U$ to be the Koopman operator $Uf=f\circ T$, which is then unitary on $L^2(X)$, 
take $V$ to be the identity, and take $p(x,y)=(p_1(x),0)$ for an integer valued polynomial $p_1$ of one variable. 
We conclude that for $f\in L^2(X)$ the averages 
\begin{equation}
 \Expectation _{0\leq n<N} f\circ T^{p_1(k_n)} 
\end{equation}
converge in $L^2(X)$. 
Invertibility is used here only to make the Koopman operator unitary; for general $T$ it is an isometry, and one would instead argue from the Herglotz representation of the positive definite sequence $m\mapsto \int_X f\cdot f\circ T^{m}\;d\mu$.
It is natural to ask for other forms of convergence. 
Kra and Shalom \cite{kra2025ergodicaverageslargeintersection} raise the question of pointwise convergence, focusing on the case of rational spectra IP sets. 
For integer Cantor sets, the question is whether, for all measure preserving systems $(X,m,T)$ and $f\in L^2(X)$, the limits 
\begin{equation}
    \lim_N \Expectation_{\mathcal C \cap [N]}  f(T^{k} x)
\end{equation}
exist for almost every $x\in X$. 
This has recently been established by Brokering Pinilla, Iosevich and Krause \cite{pinilla2026pointwiseconvergenceergodicaverages}*{Theorem 1.2}. 
For every Cantor set satisfying our standing assumption \eqref{e:standing}, every $\sigma$-finite measure preserving system $(X,\mu,T)$, and every $f\in L^p(X)$ with $2\leq p<\infty$, 
they show that the averages $\Expectation_{k\in\mathcal C\cap[N]}f(T^kx)$ converge for $\mu$-almost every $x$. 
Pointwise convergence for the polynomial averages $\Expectation_{0\leq n<N}f(T^{p(k_n)}x)$, and for the averages weighted by the sum of digits function as in Theorem \ref{t:ergodic}, remains open, as far as we know.

\section{Metric Pair Correlation}  
\label{sec;PairCorrelation}

We give sufficient conditions so that for almost all $\alpha \in [0,1]$, one has asymptotic \emph{Poissonian pair correlation} for 
$\{ \alpha k \colon k\in \mathcal{C}\}$
 for some Cantor sets $\mathcal C$.   
We recall the definition here. 

\begin{definition}  Given a sequence $(\theta_n)_{n\geq 0} \subset \mathbb R$, and $s>0$, set  
\begin{equation}
    R_2 (s,\{\theta_n \}, N) = R(s,N)
    \coloneqq 
    N^{-1} \lvert \{ 0\leq j\neq k < N \colon 
    \lVert \theta_j - \theta_k \rVert < s/N\}\rvert , 
\end{equation}
where $\lVert \cdot \rVert$ is the distance to the nearest integer. 
We say that $(\theta_n)$ has \emph{(Poissonian) pair correlation} if $R_2(s,N) \to 2s$ as $N\to \infty$, 
for every $s>0$.  
\end{definition}

Here, we focus only on pair correlation.  
A sequence with pair correlation is uniformly distributed \cites{AistleitnerLachmannPausinger2018,GrepstadLarcher2017}. And, if $\{\theta_n\}$ are iid uniformly distributed on $[0,1]$, then almost surely the sequence has pair correlation. The interest in the pair correlation for specific sequences has a deep set of connections to a range of issues.
We refer the reader to the references in \cites{MR3736514,MR4634690} for more information, and as well the paper \cite{MR4542733} that establishes, for instance,  a Weyl type result for pair and higher order correlations.

Here, using the main result from \cite{MR3736514}, we  provide a simple sufficient condition 
on a Cantor set $\mathcal C$  for the sequence 
$\{ \alpha k \colon k\in \mathcal{C}\}$
 to have pair correlation, except for a very thin set of $\alpha$.

 \begin{proposition} \label{p:pairCorrelation} 
 Assume, in addition to the standing assumption \eqref{e:standing}, that $D\subset \{0 , 1, \ldots, \lfloor (b-1)/2\rfloor \}$. 
 (As $\lvert D\rvert \geq 2$, this forces $b\geq 3$, which \eqref{e:standing} already requires.) 
 Let 
 \begin{align} \label{eED}
    E(D) = \sum_{\substack{d_1,d_2,d_3,d_4\in D \\ d_1+d_2=d_3+d_4}} 1 , 
    \qquad 
    \epsilon = \epsilon(D) \coloneqq 3 - \frac{\log E(D)}{\log \lvert D\rvert}. 
 \end{align}
 Then $\epsilon(D) >0$, and for almost all $\alpha \in [0,1]$, 
 the sequence $(\alpha k_n)_{n\geq 0}$ has pair correlation.  
 Indeed, the set of $\alpha $ for which pair correlation does not hold has Hausdorff dimension 
 at most 
 \begin{equation}
     1 - \frac{\epsilon}{ 3 + \frac{\log b}{\log \lvert D\rvert }} 
 \end{equation}
  \end{proposition}

The case of a single non-zero digit, $D=\{d\}$ with $d\neq 0$, is excluded by our standing assumption, on two counts: $0\notin D$ and $\lvert D\rvert =1$. 
It is also excluded by \eqref{eED}, which is meaningless for $\lvert D\rvert=1$. 
In that case $\mathcal C=\{d\frac{b^{\ell}-1}{b-1}\colon \ell \geq 1\}$ is lacunary, and lacunary sequences, of which $\{b^\ell\}$ is the model, are very well studied \cites{MR1924774,MR3314206,MR4164447} by much more sophisticated techniques.

The quantity $E(D)$ in \eqref{eED} is the \emph{additive energy}  of $D$.  The trivial bound is $E(D)\leq \lvert D \rvert ^{3}$, 
 since $d_1,d_2,d_3$ determine $d_4$. 
 This bound is never attained when $\lvert D\rvert\geq 2$: with $d_1=d_2=\max D$ and $d_3=0$, we have $d_1+d_2-d_3=2\max D\notin D$. 
 Hence $E(D)\leq \lvert D\rvert^3-1$, and $\epsilon(D)>0$ for every digit set satisfying our assumptions. 
 The interest of the Proposition is in the size of $\epsilon(D)$, which is larger when the additive energy of $D$ is smaller. 
We have the optimal bound $E(D) \ll \lvert D\rvert ^2$ if $D$ is \emph{Sidon}, that is there are only trivial solutions to $d_1+d_2=d_3+d_4$ for $d_1,d_2,d_3,d_4\in D$, in which case $\epsilon(D)\to 1$ as $\lvert D\rvert\to\infty$. 
This is the best possible: the trivial solutions alone give $E(D)\geq 2\lvert D\rvert ^2-\lvert D\rvert$, so that $\epsilon (D)<1$ for every $D$. 
At the other extreme, for an interval $D=\{0,1,\ldots ,m\}$ one computes $E(D)=\frac {2\lvert D\rvert ^3+\lvert D\rvert }3$, whence $\epsilon (D)=\frac{\log (3/2)}{\log \lvert D\rvert }+o(1)$, which tends to $0$. 
Since the bound on the dimension of the exceptional set is $1-\epsilon (D)\bigl(3+\frac{1}{\dim \mathcal C}\bigr)^{-1}$, the Proposition has content only for digit sets of small additive energy, and it degenerates as $\dim \mathcal C\to 0$. 
Results of Erd{\H o}s and Tur\'an \cite{MR6197} and Singer \cite{MR1501951}
 reveal that the largest cardinality of a Sidon subset of $\{0,1,\ldots,m\}$ is   $\approx m^{1/2}$.  
Also see \cite{MR4336213}.

The proof is an easy corollary to the main result of Aistleitner, Larcher and Lewko \cite{MR3736514}. 
Certainly, issues related to pair correlation would be worthy of investigation with sharper techniques.

\begin{proof}
Fix the Cantor set $\mathcal C$ as above. To apply the main result of \cite{MR3736514}, we need to estimate the additive energy of 
$A_N=\{k_0,k_1,\ldots,k_{N-1}\}$, the first $N$ elements of $\mathcal C$.  
First let $N=\lvert D\rvert^\ell$, so that $A_N=\mathcal C\cap[0,b^\ell)$. 
The condition on the allowed digits, namely that each digit is at most $(b-1)/2$, implies that adding two elements of $A_N$ does not lead to a carry digit: 
if $x=\sum_{j<\ell}x_jb^j$ and $y=\sum_{j<\ell}y_jb^j$ with $x_j,y_j\in D$, then the base $b$ digits of $x+y$ are $x_j+y_j\leq b-1$. 
Hence, for $x,y,z,w\in A_N$, we have $x+y=z+w$ if and only if $x_j+y_j=z_j+w_j$ for every $j<\ell$. 
The additive energy is therefore multiplicative, 
\begin{align}
    E(A_{\lvert D\rvert^\ell}) &= E(D) ^{\ell} 
     =  \lvert D\rvert ^{(3-\epsilon)\ell} = 
    \lvert A_{\lvert D\rvert^\ell} \rvert ^{3 - \epsilon}. 
\end{align}
For general $N$, choose $\ell$ with $\lvert D\rvert^{\ell-1}<N\leq \lvert D\rvert^\ell$. Since additive energy is monotone under inclusion, 
$E(A_N)\leq E(D)^\ell \leq \lvert D\rvert^{3}N^{3-\epsilon}$.
Observe that  $k_n$, the $n$th element of $\mathcal C$, 
satisfies $k_n \ll n ^d$, with $d= \frac{\log b}{\log \lvert D\rvert}$: if $\lvert D\rvert^{L-1}\leq n<\lvert D\rvert^L$, then $k_n<b^L\leq b\,n^d$. 
We now apply the results of Aistleitner, Larcher and Lewko \cite{MR3736514}*{Theorem 1}. 
%% TODO: confirm the theorem numbers in [ALL]; the a.e. statement and the Hausdorff dimension statement are separate results there, 
%% and the growth hypothesis a(x) << x^d is attached only to the latter. 
If $(a(x))_{x\geq1}$ is a sequence of distinct non-negative integers whose truncations satisfy $E(A_N)\ll N^{3-\epsilon}$, 
then for almost all $\alpha$ the sequence $(\alpha a(x))$ has pair correlation; and if in addition $a(x)\ll x^d$, 
the exceptional set of $\alpha$ has Hausdorff dimension at most $\frac{d+3-\epsilon}{d+3}=1-\frac{\epsilon}{d+3}$. 
(As a check on normalizations, $a(x)=x^d$ has $\epsilon=1$, and the bound reads $\frac{d+2}{d+3}$, as in \cite{MR3736514}.) 
Applied to $a(x)=k_x$, this is our conclusion.

\end{proof}

% amsrefs supplies its own bibliography style (amsrn.bst); do not set \bibliographystyle here. 
% \bib, bibdiv, biblist are defined by the amsrefs package.

\end{document}